\documentclass[a4paper,twoside,11pt]{amsart}

\reversemarginpar

\usepackage[colorlinks,hyperindex,linkcolor=blue,urlcolor=black,pdftitle={On existence of shift-type invariant subspace for polynomially bounded operators},pdfauthor={Maria F.Gamal'}]
{hyperref}

\usepackage{amsmath}
\usepackage{amsfonts}
\usepackage{amssymb}
\usepackage{amsthm}

\usepackage[english]{babel}
\usepackage{enumerate}

\usepackage{graphicx}
\usepackage{color}

\usepackage[abbrev]{amsrefs}

\numberwithin{equation}{section}

\theoremstyle{plain}
\newtheorem{theorem}{Theorem}[section]
\newtheorem{lemma}[theorem]{Lemma}
\newtheorem{corollary}[theorem]{Corollary}
\newtheorem{proposition}[theorem]{Proposition}
\newtheorem{theoremcite}{Theorem}

\theoremstyle{definition}
\newtheorem{remark}[theorem]{Remark}

\begin{document}

\title[Shift-type invariant subspaces]{On existence of shift-type invariant subspaces for polynomially bounded operators}
\author[M.F. Gamal']{Maria F. Gamal'}

\address{M.F. GAMAL', 
 St. Petersburg Branch\\ V. A. Steklov Institute 
of Mathematics\\
 Russian Academy of Sciences\\ Fontanka 27, St. Petersburg\\ 
191023, Russia  
}
\email{gamal@pdmi.ras.ru}




\begin{abstract}
A particular case of results from \cite{ker07} is as follows.  
Let the unitary asymptote of a contraction $T$ contain the bilateral shift (of 
finite or infinite multiplicity). Then there exists an invariant subspace $\mathcal M$ of $T$ 
such that $T|_{\mathcal M}$ is similar to the unilateral shift of the same multiplicity. 
The proof is based on the  Sz.-Nagy--Foias functional model for contractions. 
In the present paper this result is generalized to polynomially bounded operators, 
but in the simplest case. Namely, it is proved that if the unitary asymptote of a  polynomially bounded operator $T$ contains  the bilateral shift 
of  multiplicity $1$, then there exists an invariant subspace $\mathcal M$ of $T$ 
such that $T|_{\mathcal M}$ is similar to the unilateral shift of  multiplicity $1$. 
 The proof is based on a result from \cite{bour}.
  \end{abstract}

\maketitle

\section{Introduction}

Let $\mathcal H$ be a (complex, separable) Hilbert space, 
and let  $\mathcal L(\mathcal H)$ be the algebra of all (bounded, linear)  operators acting on  $\mathcal H$. A (closed) subspace  $\mathcal M$ of  $\mathcal H$ is called \emph{invariant} 
for an operator $T\in\mathcal L(\mathcal H)$, if $T\mathcal M\subset\mathcal M$.
 The complete lattice of all invariant subspaces of $T$ is denoted by  $\operatorname{Lat}T$. For a (closed) subspace $\mathcal M$ of a Hilbert space $\mathcal H$, by $P_{\mathcal M}$ and $I_{\mathcal M}$ the orthogonal projection from $\mathcal H$ onto $\mathcal M$ and  the identity operator on $\mathcal M$ are denoted,  respectively. 

For Hilbert spaces  $\mathcal H$ and $\mathcal K$, let    $\mathcal L(\mathcal H, \mathcal K)$ denote the space of (bound-ed, linear) 
transformations acting from $\mathcal H$ to $\mathcal K$. Suppose that $T\in\mathcal L(\mathcal H)$, $R\in\mathcal L(\mathcal K)$, 
$X\in\mathcal L(\mathcal H, \mathcal K)$, and $X$  \emph{intertwines} $T$ and $R$, that is, $XT=RX$. If $X$ is unitary, then $T$ and $R$ 
are called  \emph{unitarily equivalent}, in notation: $T\cong R$. If $X$ is invertible, that is, $X^{-1}\in\mathcal L(\mathcal K, \mathcal H)$, 
then $T$ and $R$ are called \emph{similar}, in notation: $T\approx R$.
If $X$ is a \emph{quasiaffinity}, that is, $\ker X=\{0\}$ and $\operatorname{clos}X\mathcal H=\mathcal K$, then
$T$ is called a  \emph{quasiaffine transform} of $R$, in notation: $T\prec R$. 
If  $\operatorname{clos}X\mathcal H=\mathcal K$, we write $T\buildrel d \over\prec R$. 
If $T\prec R$ and 
$R\prec T$, then $T$ and $R$ are called  \emph{quasisimilar}, in notation: $T\sim R$. Clearly, $T\prec R$ if and only if 
$R^\ast\prec T^\ast$. 

An  operator $T$ is called  \emph{power bounded}, if $\sup_{n\geq 0}\|T^n\| < \infty$. 
An operator $T$ is called  \emph{polynomially bounded}, if there exists a constant $C$
 such that 
$\|p(T)\|\leq C \max\{|p(z)|:  |z|\leq 1\} $ for every (analytic) polynomial $p$. 
The smallest such constant is called the \emph{polynomial bound} of $T$ and is denoted here by $C_{{\rm pol}, T}$. 
 
Every  polynomially bounded operator can be represented as a direct sum of an operator similar to a  singular unitary operator and of an   \emph{absolutely continuous} (a.c.) polynomially bounded operator, that is, an 
operator which  admits an $H^\infty$-functional calculus, see  \cite{mlak} or {\cite[Theorems 13, 17 and  23]{ker16}}. 
 In the present paper, absolutely continuous polynomially bounded operators are considered.  
 (Although many results on polynomially 
bounded operators that will be used in the present paper were originally  proved by Mlak \cite{mlak}, 
we will refer to \cite{ker16} for the convenience of references.)

An  operator $T$  is called a  \emph{contraction}, if $\|T\|\leq 1$. 
Every contraction is polynomially bounded with the constant $1$ by the von Neumann inequality (see, for example,  {\cite[Proposition I.8.3]{sfbk}}). Clearly, every polynomially bounded operator is power bounded. 
(It is well known that the converse is not true, see \cite{fog} for the first example of  a power bounded but not polynomially bounded operator, 
and \cite{pis} for  the first example of a  polynomially bounded operator which is not similar to a contraction.)

Let   $T\in\mathcal L(\mathcal H)$ be a power bounded operator. 
 It is easy to see that the space 
$$\mathcal H_{T,0}=\{x\in\mathcal H \ :\ \|T^nx\|\to 0\}$$ is invariant for $T$ (sf. {\cite[Theorem II.5.4]{sfbk}}). Classes $C_{ab}$, $a$, $b=0,1,\cdot$, of power bounded operators 
are defined as follows (see  {\cite[Sec. II.4]{sfbk}} and \cite{ker89}). If $\mathcal H_{T,0}=\mathcal H$, then  $T$ is  \emph{of class} $C_{0\cdot}$, while if  $\mathcal H_{T,0}=\{0\}$, then $T$ is 
 \emph{of class} $C_{1\cdot}$. Furthermore,  $T$  is \emph{of class} $C_{\cdot a}$, if $T^\ast$ is of class  $C_{a\cdot}$,  
 and $T$ is  \emph{of class} $C_{ab}$, if $T$ is of classes $C_{a\cdot}$ and $C_{\cdot b}$, $a$, $b=0,1$. 

Let $T$ and $R$ be power bounded operators, and let $T\prec R$.  
It easily follows from the definition that if $R$ is of class $C_{1\cdot}$ or of class $C_{\cdot 0}$, 
then $T$ is of class $C_{1\cdot}$ or of class $C_{\cdot 0}$, too. Clearly, any isometry is of class $C_{1\cdot}$, a unitary operator is of class $C_{11}$, and the unilateral shift is of class $C_{10}$. 

The notions of  isometric and unitary asymptotes will be used. 
Although there exist further studies of these notions (\cite{ker16}, \cite{kerntuple}), we restrict ourselves  
to the case of power bounded operators.  
For a power bounded operator $T\in\mathcal L(\mathcal H)$  the \emph{isometric asymptote} $(X_+,T_+^{(a)})$   
 can be defined using a Banach limit, see \cite{ker89}.
Here $T_+^{(a)}$ is an isometry (on a Hilbert space), and $X_+$ is the \emph{canonical intertwining mapping}: 
$X_+T=T_+^{(a)}X_+$. Recall that $\ker X_+ = \mathcal H_{T,0}$, and the range of $X_+$ is dense. 
Thus, $T\buildrel d\over\prec  T_+^{(a)}$. In particular, if $T$ is cyclic, then  $T_+^{(a)}$ is cyclic, too.
A power bounded operator $T$ is of class $C_{11}$ if and only if $T$ is quasisimilar to  a unitary operator; then 
$T_+^{(a)}$ is unitary and $T\sim T_+^{(a)}$  (see {\cite[Propositions  II.3.4 and II.5.3]{sfbk}}). 

The \emph{unitary asymptote} $(X,T^{(a)})$ of a power bounded operator $T\in\mathcal L(\mathcal H)$ is a pair where 
 $T^{(a)}\in\mathcal L(\mathcal H^{(a)})$ (here $\mathcal H^{(a)}$ is a some Hilbert space) is the minimal unitary extension of $T_+^{(a)}$, and $X$ is an extension of $X_+$. Therefore,
\begin{equation}\label{unitdense}
\bigvee_{n\geq 0}(T^{(a)})^{-n} X \mathcal H= \mathcal H^{(a)}.
\end{equation}
 In particular, if $T$ is cyclic, then $T^{(a)}$ is cyclic, too.
If $T_+^{(a)}$ is unitary, then, of course, $(X_+,T_+^{(a)})=(X,T^{(a)})$. Note that 
 $T_+^{(a)}$ and  $T^{(a)}$ are defined up to unitarily equivalence.

 Set $T_{0\cdot}=T|_{\mathcal H_{T,0}}$ and $T_{1\cdot}=P_{\mathcal H_{T,0}}T|_{\mathcal H_{T,0}}$. 
By {\cite[Lemma 1]{ker89}}, $T_{0\cdot}$ and $T_{1\cdot}$ are of  classes $C_{0\cdot}$ and 
$C_{1\cdot}$, respectively. Thus, every power bounded operator has the   triangulations of  the form 
$$\begin{pmatrix}C_{0\cdot} & * \\ \mathbb O & C_{1\cdot}\end{pmatrix} \ \ \text{ and } \ \   
\begin{pmatrix}C_{\cdot 1} & * \\ \mathbb O & C_{\cdot 0}\end{pmatrix}.$$ 
By  {\cite[Theorem 3]{ker89}}, $T^{(a)}\cong T_{1\cdot}^{(a)}$. 

Let $T$ be a power bounded operator,  let $U$ be a unitary operator, and let $T\buildrel d\over\prec U$. 
By  {\cite[Theorem 2]{ker89}}, $T^{(a)}$ contains $U$ as an orthogonal summand, that is, there exists a unitary 
operator $V$ such that  $T^{(a)}\cong U\oplus V$. On the other hand, it follows from \eqref{unitdense} that 
$T\buildrel d\over\prec T^{(a)}$ (but this relation \emph{ is not } realized by the canonical intertwining mapping in general). 

A particular case of \cite{ker07} is the following (see also {\cite[Sec. IX.3]{sfbk}} 
and references in \cite{ker07} and \cite{sfbk} for the history of a question). 
Let $T$ be a contraction, and let $T^{(a)}$ contain the bilateral shift of 
finite or infinite multiplicity as an orthogonal summand. Then there exists $\mathcal M\in \operatorname{Lat}T$ 
such that $T|_{\mathcal M}$ is similar to the unilateral shift of the same multiplicity. 
The proof is based on the Sz.-Nagy--Foias functional model for contractions, see \cite{sfbk}. 
In the present paper this result is generalized to polynomially bounded operators, 
but in the simplest case. Namely, it is proved that 
if $T$ is a polynomially bounded operator and $T^{(a)}$ contains the bilateral shift of 
 multiplicity $1$, then there exists $\mathcal M\in \operatorname{Lat}T$ 
such that $T|_{\mathcal M}$ is similar to the unilateral shift of multiplicity $1$. The proof is based on a result from \cite{bour}. 

Before formulating the main result of the paper, we introduce some notation.
 $\mathbb T$ and $\mathbb D$ denote the unit circle and the open unit disc, respectively.   
$S$ is the \emph{unilateral shift} of multiplicity $1$,  that is, $S$ is the operator of multiplication by the independent variable on the Hardy space $H^2$ on  $\mathbb T$. 
 The normalized Lebesgue measure on $\mathbb T$ is denoted by $m$.  
For a measurable set $\sigma\subset\mathbb T$ denote by $U_\sigma$ 
the operator of multiplication by the independent variable on $L^2(\sigma) := L^2(\sigma,m)$. 
Clearly, $U_{\mathbb T}$ is the \emph{bilateral shift} of multiplicity $1$. It is well known and easy to see that $S$ and
 $U_\sigma$ are a.c. contractions.

\begin{theorem}\label{thmmain} Suppose that $T$ is a cyclic a.c. polynomially bounded operator of class $C_{1\cdot}$, and  
$T^{(a)}\cong U_{\mathbb T}$. Then there exists $\mathcal M\in \operatorname{Lat}T$ 
such that $T|_{\mathcal M}\approx S$. 
\end{theorem}

\begin{corollary}\label{cormain} Suppose that $T$ is a polynomially bounded operator such that 
$T\buildrel d\over\prec  U_{\mathbb T}$. Then for every $c>0$ there exist $\mathcal M\in \operatorname{Lat}T$ 
and $W\in \mathcal L(H^2,\mathcal M)$ such that $W$ is invertible, $WS=T|_{\mathcal M}W$ and $$\|W\|\|W^{-1}\|\leq 
(1+c)\bigl(\sqrt 2(K^2+2)K C_{{\rm pol}, T}+1\bigr)\sqrt{ K^2 C_{{\rm pol}, T}^2 + 1} K  C_{{\rm pol}, T}^2.$$ 
\end{corollary}

\begin{remark} Examples of polynomially bounded operators satisfying Theorem \ref{thmmain}  and not similar to contractions can be found in \cite{gam16}. Moreover, it follows from Theorem \ref{thmmain}  and Corollary \ref{cormain} 
that if $T$ is a polynomially bounded operator such that $T\sim U_{\mathbb T}$ 
and the product of intertwining quasiaffinities is an analytic function of $U_{\mathbb T}$, 
or if $T\prec S$, then $T$ is similar to an operator constructed in  {\cite[Proposition 2.7 or Corollary 2.3]{gam16}}, respectively.
\end{remark}

\begin{remark} Let $T$ satisfy the assumption of Corollary \ref{cormain}. By  {\cite[Theorem 4]{ker89}}, 
$\mathbb T\subset \sigma(T)$. By \cite{rej}, $T$ either has a nontrivial hyperinvariant subspace or is reflexive. 
 Therefore, the result on existence of nontrivial invariant subspaces of $T$ is not new. The reflexivity of  contractions satisfying the assumptions  of Corollary \ref{cormain} is proved in  
\cite{tak87} and \cite{tak89}, see also  {\cite[Theorem IX.3.8]{sfbk}}. This proof   
can be generalized to polynomially bounded operators. We expect to give a detailed proof  later elsewhere. 
\end{remark}

\begin{remark} Suppose that $N\in\mathbb N$ and $T$ is a polynomially bounded operator such that 
$T\buildrel d\over\prec \oplus_{n=1}^N U_{\mathbb T}$. It is possible to prove by induction that  
 there exists $\mathcal M\in \operatorname{Lat}T$ such that $T|_{\mathcal M}\approx \oplus_{n=1}^N S$. 
In the case of infinite sum the question remains open.
\end{remark}

 The paper is organized as follows. In Sec. 2, Bourgain's result \cite{bour} is cited, and its corollaries which hold true
for arbitrary a.c. polynomially bounded operators are given. In Sec. 3 and 4, auxiliary results for some functions and for 
operators intertwined with unitaries, respectively, are given. The main part of the paper is Sec. 5,  where 
Theorem \ref{thmmain} is proved. In Sec. 6 Corollary \ref{cormain} is proved.

\section {Bourgain's result}

 $\text{\bf A}$ denotes the disc algebra. For a positive finite Borel measure $\mu$ on $\mathbb T$ set
 $P^2(\mu)=\operatorname{clos}_{L^2(\mu)}\text{\bf A}$, 
and denote  by $S_\mu$ the operator of multiplication by the independent variable on $P^2(\mu)$. 

\begin{theoremcite}[\cite{bour}]\label{mainbour} There exists a universal constant $K$ with the following property. 
Let $\mathcal H$ be a Hilbert space, 
and let $\text{\bf W}\in\mathcal L(\text{\bf A},\mathcal H)$ (that is, $\text{\bf W}$ is a bounded linear transformation 
from $\text{\bf A}$ to $\mathcal H$).   Then there exists 
a positive Borel measure $\mu$ on $\mathbb T$ such that $\mu(\mathbb T)=1$ and 
$$\|\text{\bf W}\varphi\|\leq K \|\text{\bf W}\|\Bigl(\int_{\mathbb T}|\varphi|^2\text{\rm d}\mu\Bigr)^{1/2} 
\ \ \text{ for every }
\varphi\in \text{{\bf A}}.$$
\end{theoremcite}

The following corollary  is  {\cite[Lemma 2.1]{bercpr}} emphasized for a.c. polynomially bounded operators.

\begin{corollary}\label{tha}  
Suppose that $\mathcal H$ is a Hilbert space, $T\in\mathcal L(\mathcal H)$ is an a.c. polynomially bounded operator, 
$C_{{\rm pol}, T}$ is the polynomial bound of $T$, and $x\in\mathcal H$. Then there exist $\psi\in L^2(\mathbb T,m)$ 
such that $\int_{\mathbb T}|\psi|^2\text{\rm d}m\leq 1$ and 
$$\|\varphi(T)x\|\leq K C_{{\rm pol}, T}\|x\|\Bigl(\int_{\mathbb T}|\varphi|^2|\psi|^2\text{\rm d}m\Bigr)^{1/2} \ \ \text{ for every }
\varphi\in H^\infty$$
and $W\in\mathcal L(P^2(|\psi|^2m),\mathcal H)$ such that 
\begin{align*}& WS_{|\psi|^2m}=TW, \ \ \ \|W\|\leq K C_{{\rm pol}, T}\|x\| \\ \text{ and } \ \ \ 
& W\varphi=\varphi(T)x \ \ \text{ for every }\varphi\in H^\infty. \end{align*}
\end{corollary}

\begin{proof}
Define $\text{\bf W}\in\mathcal L(\text{\bf A},\mathcal H)$ by the formula 
$$\text{\bf W}\varphi=\varphi(T)x, \ \ \varphi\in \text{\bf A}.$$ 
Let $\mu$ be the measure from Theorem \ref{mainbour}. Then $\mu=|\psi|^2m+\mu_s$, 
where $\psi\in L^2(\mathbb T,m)$ and $\mu_s$ is a positive Borel measure on $\mathbb T$ singular with respect to $m$. 
We have  $$P^2(\mu)=P^2(|\psi|^2m)\oplus L^2(\mu_s)$$
({\cite[Proposition III.12.3]{conw}} or {\cite[Corollary A.2.2.1]{nik}}), 
$S_{\mu_s}$ is a singular unitary operator,  
and $\text{\bf W}$ has a continuous extension on $P^2(\mu)$  denoted by $W$. Clearly, $WS_\mu=TW$.
Therefore, $W|_{L^2(\mu_s)}S_{\mu_s}=TW|_{L^2(\mu_s)}$. Since $T$ is a.c., $W|_{L^2(\mu_s)}=\mathbb O$
( {\cite[Proposition 15]{ker16}} or \cite{mlak}). 
Thus, $\mu$ can be replaced by $|\psi|^2m$. 

Let $ \varphi\in H^\infty$. Set $\varphi_r(\zeta)=\varphi(r\zeta)$, $\zeta\in\operatorname{clos}\mathbb D$, $0<r<1$. 
Then $ \varphi_r\in \text{\bf A}$, and $\varphi_r\to\varphi$ when $r\to 1$ in the weak-$\ast$ topology in $H^\infty$ and 
in the norm in $L^2(|\psi|^2m)$ simultaneously. The conclusion of the corollary follows from these convergences.
\end{proof}

\begin{remark}\label{remc1} 
For contractions $T$,  $C_{{\rm pol}, T}=1$, and  Corollary \ref{tha} is proved in  {\cite[Lemma 3]{berctak}} with $K=1$. 
The proof is based on the existence of isometric dilations for contractions, see {\cite[Theorem I.4.1]{sfbk}}. 
\end{remark}

\begin{remark}\label{h2} Let $\psi\in L^2(\mathbb T)$. If $\int_{\mathbb T}\log|\psi|\text{\rm d}m=-\infty$, 
then $P^2(|\psi|^2m)=L^2(|\psi|^2m)$. 
 If $\int_{\mathbb T}\log|\psi|\text{\rm d}m>-\infty$, we accept that $\psi$ is an outer function. Then 
 $$P^2(|\psi|^2m)=\frac {H^2}{\psi}=\Bigl\{\frac{h}{\psi}, \ h\in H^2\Bigr\}, 
\ \ \Bigl\|\frac{h}{\psi}\Bigr\|_{P^2(|\psi|^2m)}=\|h\|_{H^2}, \ \ h\in H^2,$$
 see  {\cite[Ch. III.12, VII.10]{conw}} or  {\cite[Ch. A.4.1]{nik}}.
\end{remark}

\begin{lemma}\label{lemfinit} 
Suppose that $\mathcal H$ is a Hilbert space, $T\in\mathcal L(\mathcal H)$ is an a.c. polynomially bounded operator, 
$C_{{\rm pol}, T}$ is the polynomial bound of $T$. Given $N\in\mathbb N$ and  $\{\varphi_n\}_{n=1}^N\subset H^\infty$, 
 define $$A\in\mathcal L(\mathcal H,\oplus_{n=1}^N\mathcal H), 
\ \ \ Ax=\oplus_{n=1}^N\varphi_n(T)x, \ \ x\in\mathcal H.$$ 
Then $$\|A\|\leq KC_{{\rm pol}, T}
\mathop{\text{\rm ess\,sup}}_{\zeta\in\mathbb T}
\Bigr(\sum_{n=1}^N|\varphi_n(\zeta)|^2\Bigr)^{1/2}.$$
\end{lemma}

\begin{proof}
 For $x\in\mathcal H$, let $\psi_x\in L^2(\mathbb T)$ be the function from Corollary \ref{tha} (applied to $T$). 
We have \begin{align*}\|Ax\|^2&=\sum_{n=1}^N\|\varphi_n(T)x\|^2\\&
\leq\sum_{n=1}^NK^2 C_{{\rm pol}, T}^2\|x\|^2\int_{\mathbb T}|\varphi_n|^2|\psi_x|^2\text{\rm d}m\\&
= K^2 C_{{\rm pol}, T}^2\|x\|^2\int_{\mathbb T}\sum_{n=1}^N|\varphi_n|^2|\psi_x|^2\text{\rm d}m\\&
\leq K^2 C_{{\rm pol}, T}^2\|x\|^2\mathop{\text{\rm ess\,sup}}_{\mathbb T}\sum_{n=1}^N|\varphi_n|^2\int_{\mathbb T}|\psi_x|^2\text{\rm d}m\\&
\leq \Bigl(K^2 C_{{\rm pol}, T}^2\mathop{\text{\rm ess\,sup}}_{\mathbb T}\sum_{n=1}^N|\varphi_n|^2\Bigr)\cdot\|x\|^2. \qedhere
\end{align*}
\end{proof}

\begin{theorem}\label{thinfinit} Suppose that $\mathcal H$ is a Hilbert space, 
$T\in\mathcal L(\mathcal H)$ is an a.c. polynomially bounded operator, 
$C_{{\rm pol}, T}$ is the polynomial bound of $T$. Then 
$$\Bigl\|\sum_{n\in\mathbb N}\varphi_n(T)x_n\Bigr\|\leq K C_{{\rm pol}, T}
\mathop{\text{\rm ess\,sup}}_{\zeta\in\mathbb T}
\Bigr(\sum_{n\in\mathbb N}|\varphi_n(\zeta)|^2\Bigr)^{1/2}\Bigr(\sum_{n\in\mathbb N}\|x_n\|^2\Bigr)^{1/2}$$
for every $ \{\varphi_n\}_{n\in\mathbb N}\subset H^\infty$ and every $ \{x_n\}_{n\in\mathbb N}\subset \mathcal H$ 
such that the right part of the above inequality is finite. 
\end{theorem}

\begin{proof} Let $\{\varphi_n\}_{n\in\mathbb N}\subset H^\infty$ be such that 
$$\mathop{\text{\rm ess\,sup}}_{\zeta\in\mathbb T}\sum_{n\in\mathbb N}|\varphi_n(\zeta)|^2:=a<\infty.$$
Let $N\in\mathbb N$. Let $A$ be the transformation from Lemma \ref{lemfinit} applied to  
$T^\ast$ and $\{\widetilde{\varphi_n}\}_{n=1}^N$, 
where $\widetilde{\varphi}(\zeta)=\overline{\varphi(\overline\zeta)}$ ($\zeta \in\mathbb D$, $\varphi\in H^\infty$). 
Taking into account that $\|A\|=\|A^\ast\|$ and $\varphi(T)^\ast=\widetilde{\varphi}(T^\ast)$ 
({\cite[Proposition 24]{ker16}} or \cite{mlak}), 
we obtain that 
\begin{align*}\Bigl\|\sum_{n=1}^N\varphi_n(T)x_n\Bigr\|&=\|A^\ast\bigl(\oplus_{n=1}^Nx_n\bigr)\| \\ &
\leq K C_{{\rm pol}, T}\mathop{\text{\rm ess\,sup}}_{\zeta\in\mathbb T}
\Bigr(\sum_{n=1}^N|\varphi_n(\zeta)|^2\Bigr)^{1/2}\Bigr(\sum_{n=1}^N\|x_n\|^2\Bigr)^{1/2}
\\&
\leq K C_{{\rm pol}, T}a^{1/2}\Bigr(\sum_{n=1}^N\|x_n\|^2\Bigr)^{1/2}.\end{align*}
Therefore, if $\sum_{n\in\mathbb N}\|x_n\|^2<\infty$, then $\sum_{n\in\mathbb N}\varphi_n(T)x_n$ converges, and the theorem follows. 
 \end{proof}

\begin{lemma}\label{lemmin}  
Suppose that $\mathcal H$ is a Hilbert space, $T\in\mathcal L(\mathcal H)$ is an a.c. polynomially bounded operator, 
$x\in\mathcal H$,  $\psi_k\in L^2(\mathbb T,m)$, $k=1,2$, are  
such that $\int_{\mathbb T}\log|\psi_k|\text{\rm d}m>-\infty$, and 
$$\|\varphi(T)x\|\leq \Bigl(\int_{\mathbb T}|\varphi|^2|\psi_k|^2\text{\rm d}m\Bigr)^{1/2} \ \ \text{ for every }
\varphi\in H^\infty, \ \ \ k=1,2.$$
Set $|\psi(\zeta)|=\min(|\psi_1(\zeta)|,|\psi_2(\zeta)|)$, $\zeta\in\mathbb T$.
Then for every $0<\varepsilon<1$ there exists an inner function $\omega\in H^\infty$ such that 
$$\|\omega(T)\varphi(T)x\|\leq \frac{(2(1+\varepsilon^2))^{1/2}}{1-\varepsilon}
\Bigl(\int_{\mathbb T}|\varphi|^2|\psi|^2\text{\rm d}m\Bigr)^{1/2} 
\ \ \text{ for every } \ \varphi\in H^\infty.$$
\end{lemma}

\begin{proof} We have $\mathbb T=\tau_1\cup\tau_2$, where $\tau_1\cap\tau_2=\emptyset$ and $|\psi|=|\psi_k|$ a.e. on $\tau_k$, 
 $k=1,2$. There exist outer functions $\eta_k\in H^\infty$,  $k=1,2$, such that 
$$|\eta_1|=\begin{cases}1  & \text{ on } \tau_1, \\
\varepsilon\frac{|\psi_2|}{|\psi_1|} & \text{ on } \tau_ 2,\end{cases} \ \ \  \ \text{ and } \ \ \ \  
|\eta_2|=\begin{cases}\varepsilon\frac{|\psi_1|}{|\psi_2|} & \text{ on } \tau_1, \\ 1  & \text{ on } \tau_2.
\end{cases}$$
Clearly, $$1-\varepsilon\leq|\eta_1+\eta_2|\leq 1+\varepsilon \ \  \ \text{ a.e. on }\ \mathbb T.$$
Set $\omega\eta=\eta_1+\eta_2$, where $\omega$, $\eta\in H^\infty$, $\omega$ is inner, and $\eta$ is outer.
Let $\varphi\in H^\infty$. We have 
\begin{align*} \|\omega(T)\varphi(T)x\|^2 &=\Bigl\|\omega(T)\eta(T)\Bigl(\frac{\varphi}{\eta}\Bigr)(T)x\Bigr\|^2=
\Bigl\|(\eta_1+\eta_2)(T)\Bigl(\frac{\varphi}{\eta}\Bigr)(T)x\Bigr\|^2  \\& 
\leq 2\Bigl(\Bigl\|\eta_1(T)\Bigl(\frac{\varphi}{\eta}\Bigr)(T)x\Bigr\|^2 + 
\Bigl\|\eta_2(T)\Bigl(\frac{\varphi}{\eta}\Bigr)(T)x\Bigr\|^2 \Bigr) \\& 
\leq 2\Bigl(\int_{\mathbb T}|\eta_1|^2\Bigl|\frac{\varphi}{\eta}\Bigr|^2|\psi_1|^2\text{\rm d}m 
+ \int_{\mathbb T}|\eta_2|^2\Bigl|\frac{\varphi}{\eta}\Bigr|^2|\psi_2|^2\text{\rm d}m \Bigr)\\& 
\leq \frac{2}{(1-\varepsilon)^2}
\Bigl(\int_{\mathbb T}|\eta_1|^2|\varphi|^2|\psi_1|^2\text{\rm d}m 
+ \int_{\mathbb T}|\eta_2|^2|\varphi|^2|\psi_2|^2\text{\rm d}m \Bigr)\\&
= \frac{2}{(1-\varepsilon)^2}
\Bigl(\int_{\tau_1}|\varphi|^2|\psi|^2\text{\rm d}m +\varepsilon^2\int_{\tau_2}|\varphi|^2|\psi|^2\text{\rm d}m \\&
\ \  \ \ 
+ \varepsilon^2\int_{\tau_1}|\varphi|^2|\psi|^2\text{\rm d}m +\int_{\tau_2}|\varphi|^2|\psi|^2\text{\rm d}m \Bigr)\\&
=\frac{2(1+\varepsilon^2)}{(1-\varepsilon)^2}\int_{\mathbb T}|\varphi|^2|\psi|^2\text{\rm d}m.\qedhere\end{align*}
\end{proof}

\begin{remark}Applying Theorem \ref{thinfinit}, it is possible to prove an analog of Lemma \ref{lemmin} for a finite 
family $\{\psi_k\}_{k=1}^N$ with a constant  not depended on $N$, and for a countable family 
 $\{\psi_k\}_{k\in\mathbb N}$ under additional assumption $\int_{\mathbb T}\log|\psi|\text{\rm d}m>-\infty$,  
where $|\psi(\zeta)|=\inf_{k\in\mathbb N}|\psi_k(\zeta)|$, $\zeta\in\mathbb T$. We do not prove these statements, because we will not apply them.
\end{remark}

\section{Preliminaries: function theory}

\begin{lemma}\label{lempsit} Suppose that $\psi\in H^2$ is an outer function. For $t>0$ let $\psi_t$ be an outer function such that 
$|\psi_t|=\max(|\psi|, t)$. Then $$\Bigl\|\frac{\psi}{\psi_t}-1\Bigl\|_{H^2}\to 0 \ \ \text{ when } t\to 0.$$ 
\end{lemma}

\begin{proof} For $0<t<1$ set $\sigma_t=\{\zeta\in\mathbb T\ :\ |\psi(\zeta)|< t\}$. 
Since $|\psi|\neq 0$ a.e. on $\mathbb T$, $m(\sigma_t)\to 0$  when $ t\to 0$.
Therefore, 
\begin{equation}\label{psinorm} 1\geq \Bigl\|\frac{\psi}{\psi_t}\Bigr\|_{H^2}^2\geq m(\mathbb T\setminus\sigma_t)\to 1 \ \ \text{ when } t\to 0.\end{equation} 
Furthermore, 
$$ 0\leq(-\log t)m(\sigma_t)\leq - \int_{\sigma_t}\log|\psi|\text{\rm d}m \to 0\ \ \text{ when } t\to 0,$$
because  $m(\sigma_t)\to 0$ and  $\int_{\mathbb T}\log|\psi|\text{\rm d}m>-\infty$.
Therefore, 
$$ \int_{\sigma_t}\log\frac{|\psi|}{t}\text{\rm d}m \to 0 \ \ \ \text{ when } t\to 0.$$
 Consequently, 
\begin{equation}\label{psiz} \frac{\psi}{\psi_t}(0)=
\exp\Bigl( \int_{\sigma_t}\log\frac{|\psi|}{t}\text{\rm d}m\Bigr)  \to 1
 \ \ \text{ when } t\to 0.
\end{equation}
By \eqref{psinorm} and \eqref{psiz}, 
 \begin{equation}\label{psito}\Bigl\|\frac{\psi}{\psi_t} - \frac{\psi}{\psi_t}(0)\Bigr\|_{H^2}^2 = 
\Bigl\|\frac{\psi}{\psi_t}\Bigr\|_{H^2}^2-\Bigl|\frac{\psi}{\psi_t}(0)\Bigr|^2 \to 0
 \ \ \text{ when } t\to 0.\end{equation} 
By \eqref{psiz} and \eqref{psito}, 
$$\Bigl\|\frac{\psi}{\psi_t}-1\Bigl\|_{H^2}\leq \Bigl\|\frac{\psi}{\psi_t} - \frac{\psi}{\psi_t}(0)\Bigr\|_{H^2} + 
\Bigl|\frac{\psi}{\psi_t}(0)-1\Bigr| \to 0 \ \ \text{ when } t\to 0.\qedhere$$ 
 \end{proof}

\begin{proposition}\label{propalphasigma}
Suppose that  $\sigma\subset\mathbb T$ is a measurable set, $\{\tau_n\}_{n\in\mathbb N}$ is a family of  measurable sets 
such that $\sigma=\cup_{n\in\mathbb N}\tau_n$ and $\tau_n\cap\tau_k=\emptyset$, if $n\neq k$. Suppose that $C>0$ and 
$\{\xi_n\}_{n\in\mathbb N}$ is a family of positive numbers such that $\sum_{n\in\mathbb N}\xi_n^2<\infty$. 
Suppose that $\{\eta_n\}_{n\in\mathbb N}$,  $\{\varphi_n\}_{n\in\mathbb N}$, and  $\{\psi_n\}_{n\in\mathbb N}$ are families 
 of  functions from $L^2(\sigma)$ such that 
\begin{equation*}\label{etasigma} |\eta_n|\leq\begin{cases}1  & \text{ on } \tau_n, \\
\xi_n & \text{ on } \sigma\setminus\tau_ n,\end{cases}\end{equation*} 
and $|\varphi_n|\leq C|\psi_n|$ a.e. on $\sigma$. For every $ t>0$ and $n\in\mathbb N$ let $\psi_{nt}\in L^2(\sigma)$ 
be such that $|\psi_n|\leq|\psi_{nt}|$ a.e. on $\sigma$ and 
$$\Bigl\|\frac{\psi_n}{\psi_{nt}}-1\Bigr\|_{L^2(\sigma)}\to 0 \ \ \text{ when } t\to 0 \ \ \text{ for every } n\in\mathbb N.$$ 
Set $$\alpha_t=\sum_{n\in\mathbb N}\eta_n^2\frac{\varphi_n}{\psi_{nt}} \ \ \text{ and } \ \ 
\alpha=\sum_{n\in\mathbb N}\eta_n^2\frac{\varphi_n}{\psi_n}.$$
Then $\sup_{t>0}\|\alpha_t\|_{L^\infty(\sigma)}<\infty$ and $\|\alpha_t-\alpha\|_{L^2(\sigma)}\to 0$ when $t\to 0$.
\end{proposition}

\begin{proof} Clearly, $|\frac{\psi_n}{\psi_{nt}}-1|\leq 2$ a.e. on $\sigma$. Therefore, 
$$ |\alpha_t-\alpha|\leq\sum_{n\in\mathbb N}|\eta_n|^2\Bigl|\frac{\varphi_n}{\psi_n}\Bigr|
\Bigl|\frac{\psi_n}{\psi_{nt}}-1\Bigr|\leq 2C\Bigl(1+\sum_{k\neq n}\xi_k^2\Bigr)\text{ on } \tau_n.$$
Take $\varepsilon >0$. There exists $N\in\mathbb N$ such that 
$$ \sum_{n\geq N+1}\xi_n^2<\varepsilon \ \ \text{ and } \ \ m\Bigl(\bigcup_{n\geq N+1}\tau_n\Bigr)<\varepsilon.$$
We have $$\int_{\bigcup_{n\geq N+1}\tau_n}|\alpha_t-\alpha|^2\text{\rm d}m
\leq  4C^2\Bigl(1+\sum_{n\in\mathbb N}\xi_n^2\Bigr)^2\varepsilon.$$
Let $1\leq n\leq N$. We have 
$$  |\alpha_t-\alpha|  \leq C\Bigl(\Bigl|\frac{\psi_n}{\psi_{nt}}-1\Bigr| 
+ \sum_{k\neq n,k\leq N}\xi_k^2\Bigl|\frac{\psi_k}{\psi_{kt}}-1\Bigr| +  2\sum_{n\geq N+1}\xi_n^2\Bigr)
 \ \text{ on } \tau_n.$$
Therefore, 
\begin{align*}  |\alpha_t&-\alpha|^2   \leq  C^2\Bigl(2N\Bigl(\Bigl|\frac{\psi_n}{\psi_{nt}}-1\Bigr|^2 
+ \sum_{k\neq n,k\leq N}\xi_k^4\Bigl|\frac{\psi_k}{\psi_{kt}}-1\Bigr|^2\Bigr) + 8\varepsilon^2\Bigr) \\ &\leq 
2N C^2\max\bigl(1,\sup_{k\in\mathbb N}\xi_k^4\bigr)\sum_{k=1}^N\Bigl|\frac{\psi_k}{\psi_{kt}}-1\Bigr|^2 + 8C^2\varepsilon^2\ \text{ on } \bigcup_{n=1}^N\tau_n.
\end{align*}
There exists $t_\varepsilon>0$ such that  
$$\int_\sigma\Bigl|\frac{\psi_n}{\psi_{nt}}-1\Bigr|^2\text{\rm d}m \leq
\frac{\varepsilon}{2N^2 C^2\max(1,\sup_{k\in\mathbb N}\xi_k^4)}$$
 for every $ 0<t<t_\varepsilon $  and $ n=1, \ldots, N.$
Therefore,
$$\int_{\bigcup_{n=1}^N\tau_n}|\alpha_t-\alpha|^2\text{\rm d}m \leq\varepsilon+8C^2\varepsilon^2 
\ \ \text{ for every } 0<t<t_\varepsilon.\qedhere$$
\end{proof}

\section{Preliminaries: operators intertwined with unitaries}

\begin{lemma}\label{lemfpsi} Suppose that $T\in\mathcal L(\mathcal H)$ is an a.c. polynomially bounded operator, 
$C_{{\rm pol}, T}$ is the polynomial bound of $T$,  $x\in\mathcal H$, and $\psi\in L^2(\mathbb T)$ 
is from Corollary \ref{tha}. Furthermore, suppose that  $\sigma\subset\mathbb T$ is a measurable set, $X\in\mathcal L(\mathcal H,L^2(\sigma))$,  $XT=U_\sigma X$, 
and $f=Xx$. Then 
$$|f|\leq \|X\|KC_{{\rm pol}, T}\|x\||\psi| \ \ \text{ a.e. on }\ \sigma.$$\end{lemma}

\begin{proof} We have 
\begin{align*}\int_\sigma|\varphi|^2|f|^2\text{\rm d}m & =\|\varphi(U_\sigma)Xx\|^2=
\|X\varphi(T)x\|^2\leq\|X\|^2\|\varphi(T)x\|^2 \\ &
\leq\|X\|^2K^2 C_{{\rm pol}, T}^2\|x\|^2\int_{\mathbb T}|\varphi|^2|\psi|^2\text{\rm d}m\end{align*}
for every $\varphi\in H^\infty$. It remains  to apply  standard reasons based on the fact that for every measurable set $\tau\subset\mathbb T$ and every $a$,$b>0$ there exists $\varphi\in H^\infty$ such that $|\varphi|=a$ a.e. on $\tau$ and $|\varphi|=b$ a.e. on $\mathbb T\setminus\tau$.
\end{proof}

\begin{lemma}\label{asymp1} Suppose that $\mathop{\text{\rm Lim}}$ is a Banach limit, $\mathcal H$ is a Hilbert space, 
$T\in\mathcal L(\mathcal H)$ is a power bounded operator, $X$ is the canonical intertwining mapping constructed 
using $\mathop{\text{\rm Lim}}$. If $\mathcal M\in\operatorname{Lat}T$ is such that $T|_{\mathcal M}$ is not of class $C_{0\cdot}$, 
then $\|X |_{\mathcal M}\|\geq 1$.\end{lemma}

\begin{proof} We do not recall the construction of the canonical intertwining mapping from \cite{ker89} here. 
We recall only that $\|Xx\|^2=\mathop{\text{\rm Lim}}_n\|T^nx\|^2$ for every $x\in \mathcal H$. By assumption, there exists 
 $x\in\mathcal M$ such that $\inf_{n\geq 0}\|T^nx\|>0$. Therefore, $\|Xx\|>0$. Since $(T^{(a)})^nX=XT^n$ and 
$T^{(a)}$ is unitary,      we have 
$$\|Xx\|=\|( T^{(a)})^nX x\|=\|XT^nx\|\leq\|X |_{\mathcal M}\|\|T^nx\| \ \ \text{ for every } \ n\geq 0.$$ Consequently, 
$$\|Xx\|^2\leq\|X|_{\mathcal M}\|^2 \mathop{\text{\rm Lim}}_n\|T^nx\|^2=\|X |_{\mathcal M}\|^2\|Xx\|^2.$$ 
Since $\|Xx\|>0$, the conclusion of the lemma follows.
\end{proof}

\begin{lemma}\label{lemast} Suppose that  $T\in\mathcal L(\mathcal H)$ is an a.c.  polynomially bounded operator, 
$C_{{\rm pol}, T}$ is the polynomial bound of $T$,  $X_{\ast}$ is the canonical intertwining mapping 
for $T^\ast$ constructed using a Banach limit.  Let $X$ be any transformation acting from $\mathcal H$ to some Hilbert space.
For $x\in\mathcal H$ let $\psi\in L^2(\mathbb T)$ be a function from Corollary \ref{tha}.  
If $\int_{\mathbb T}\log|\psi|\text{\rm d}m=-\infty$, then $$\|Xx\|\leq KC_{{\rm pol}, T}\|XX_*^*\|\|x\|.$$
\end{lemma}

\begin{proof} Denote by $\mathcal H_*$ the space on which $(T^*)^{(a)}$ acts. 
Let $W$ be from Corollary \ref{tha}. Since $S_{|\psi|^2m}$ is unitary (see Remark \ref{h2}), 
there exists $B_\ast\in\mathcal L(\mathcal H_*, P^2(|\psi|^2m))$ such that 
 $W^\ast=B_\ast X_*$ and $\|B_\ast\|\leq \|W^\ast\|$ (see  {\cite[Theorem 2]{ker89}}).
 Set $u=B_\ast ^\ast\text{\bf 1}$, 
where $\text{\bf 1}(\zeta)=1$ ($\zeta\in\mathbb T$). 
We have \begin{align*} \|u\|\leq \| B_\ast ^\ast\|\|\text{\bf 1}\|_{P^2(|\psi|^2m)}\leq \|W\|(\int_{\mathbb T} |\psi|^2\text{\rm d}m)^{1/2}
\leq  KC_{{\rm pol}, T}\|x\|,\end{align*}
because $\int_{\mathbb T} |\psi|^2\text{\rm d}m\leq 1$.
Taking into account that 
$$Xx=XW\text{\bf 1}=XX_{\ast}^\ast B_\ast ^\ast\text{\bf 1}=XX_{\ast}^\ast u,$$ 
we conclude  that $\|Xx\|\leq \|XX_*^*\|\|u\|\leq \|XX_*^*\|KC_{{\rm pol}, T}\|x\|.$
\end{proof}

\begin{proposition}\label{proptau1}
Suppose that  $T\in\mathcal L(\mathcal H)$ is an a.c. polynomially bounded operator, 
 $\sigma\subset\mathbb T$ is a measurable set, 
$X\in\mathcal L(\mathcal H,L^2(\sigma))$,  $XT=U_\sigma X$, 
$x\in\mathcal H$ is such that $\vee_{n\geq 0}T^nx=\mathcal H$, $f=Xx$, 
  $\psi_0\in L^2(\mathbb T)$, and $\int_{\mathbb T}|\psi_0|^2\text{\rm d}m\leq 1$. 

 Furthermore, suppose that 
  $0<\varepsilon,\delta<1$, $\tau\subset\sigma$ is a measurable set, $m(\tau)>0$, 
$$ 0\neq u\in\mathcal H \ \text{ is such that } \ Xu\in L^2(\tau) \ \text{and } \ \|Xu\|\geq\delta\|u\|.$$
Set $$\delta_0=\frac{\delta\|u\|-\varepsilon\|X\|}{\|u\|+\varepsilon} 
\ \ \text{ and } \ \ 
\gamma_0=\Bigl(1-\frac{\varepsilon^2\|X\|^2}{(\|u\|-\varepsilon)^2\delta_0^2}\Bigr)^{1/2}$$
(we assume that $\varepsilon$ is so small that $\gamma_0$,$\delta_0>0$).
Let $0<\gamma<\gamma_0$. By assumption on $x$, there exists $\varphi_0\in H^\infty$ such that 
$\|\varphi_0(T)x-u\|\leq\varepsilon$. Then
$$\tau\not\subset\{\zeta\in\sigma \ :\ |\varphi_0(\zeta)f(\zeta)|\leq\gamma\delta_0\|\varphi_0(T)x\||\psi_0(\zeta)|\}.$$
\end{proposition}

\begin{proof} We have $\varphi_0f=X\varphi_0(T)x$, and  
$$(\varphi_0f)|_{\sigma\setminus\tau}=P_{L^2(\sigma\setminus\tau)}X\varphi_0(T)x=
P_{L^2(\sigma\setminus\tau)}X(\varphi_0(T)x-u).$$
Therefore, $$\int_{\sigma\setminus\tau}|\varphi_0f|^2\text{\rm d}m\leq\|X\|^2\|\varphi_0(T)x-u\|^2\leq\varepsilon^2\|X\|^2.$$
Also, \begin{align*}&\|X\varphi_0(T)x\|\geq\|Xu\|-\varepsilon\|X\|\geq\delta\|u\|-\varepsilon\|X\| \\
& \text{and } \ \ 
\delta_0\|\varphi_0(T)x\|\leq(\|u\|+\varepsilon)\delta_0.\end{align*}
Thus, $$\|X\varphi_0(T)x\|\geq\delta_0\|\varphi_0(T)x\|.$$

Let us assume that $$\tau\subset\{\zeta\in\sigma \ :\ |\varphi_0(\zeta)f(\zeta)|\leq\gamma\delta_0\|\varphi_0(T)x\||\psi_0(\zeta)|\}.$$
Then 
\begin{align*}\delta_0^2\|\varphi_0(T)x\|^2 & \leq\|X\varphi_0(T)x\|^2=\int_\sigma|\varphi_0f|^2\text{\rm d}m \\ &
=\int_{\sigma\setminus\tau}|\varphi_0f|^2\text{\rm d}m + 
\int_\tau|\varphi_0f|^2\text{\rm d}m\\&\leq\varepsilon^2\|X\|^2 + \int_\tau\gamma^2\delta_0^2\|\varphi_0(T)x\|^2|\psi_0|^2\text{\rm d}m \\ &
\leq \varepsilon^2\|X\|^2 + \gamma^2\delta_0^2\|\varphi_0(T)x\|^2\end{align*}
(because $\int_\tau|\psi_0|^2\text{\rm d}m\leq\int_{\mathbb T}|\psi_0|^2\text{\rm d}m\leq 1$). 
Thus, $$1-\gamma^2\leq\frac{\varepsilon^2\|X\|^2}{\delta_0^2\|\varphi_0(T)x\|^2}
\leq\frac{\varepsilon^2\|X\|^2}{\delta_0^2(\|u\|-\varepsilon)^2},$$ a contradiction with  assumption on $\gamma$.
\end{proof}

\begin{proposition}\label{proptau2}Suppose that  $T\in\mathcal L(\mathcal H)$ is an a.c. polynomially bounded operator, 
   $\sigma\subset\mathbb T$ is a measurable set, $T^{(a)}\cong U_\sigma$, and 
 $X$ is the canonical intertwining mapping for $T$ and $U_\sigma$ constructed using a Banach limit.
Suppose that $V\in\mathcal L(\mathcal K)$ is a unitary operator, $Y\in\mathcal L(\mathcal K,\mathcal H)$ 
is such that $YV=TY$ and $\operatorname{clos}XY\mathcal K=L^2(\sigma)$.

 Furthermore, suppose that 
$x\in\mathcal H$ and $\{\varphi_n\}_{n\in\mathbb N}\subset H^\infty$ are  such that $\{\varphi_n(T)x\}_{n\in\mathbb N}$ is dense in $\mathcal H$. Set  $f=Xx$.  Suppose that  $\{\psi_n\}_{n\in\mathbb N}\subset L^2(\mathbb T)$ are  such that
$\int_{\mathbb T}|\psi_n|^2\text{\rm d}m\leq 1$ for all  $n\in\mathbb N$.  
For $n\in\mathbb N$ and $0<c<1$ set 
$$\sigma_{n,c}=\{\zeta\in\sigma\ :\ |\varphi_n(\zeta)f(\zeta)|
\leq c\|\varphi_n(T)x\||\psi_n(\zeta)|\}. $$

For $0<\varepsilon_0,\delta_0,\gamma_0<1$ set 
\begin{align*}\mathbb N_{\varepsilon_0,\delta_0}=\{n\in\mathbb N \ :
\  & \|X\varphi_n(T)x\|\geq\delta_0\|\varphi_n(T)x\| \\  & 
\ \text{ and } \ 
1-\varepsilon_0\leq\|\varphi_n(T)x\|\leq 1+\varepsilon_0\}\end{align*}
and $$\tau_{\varepsilon_0,\delta_0,\gamma_0}=
\bigcap_{n\in\mathbb N_{\varepsilon_0,\delta_0}}\sigma_{n,\gamma_0\delta_0}.$$

Finally, let $0<\varepsilon_{00}, \delta_0,\gamma_0<1$ be given. Then there exists 
$0<\varepsilon_0\leq\varepsilon_{00}$ 
such that $$m(\tau_{\varepsilon_0,\delta_0,\gamma_0})=0.$$
\end{proposition}

\begin{proof} Set $\mathcal E=\ker XY$. Then $\mathcal E\in\operatorname{Lat}V$ 
and $XY|_{\mathcal K\ominus\mathcal E}$ realizes the relation
\begin{equation}\label{congu} P_{\mathcal K\ominus\mathcal E}V|_{\mathcal K\ominus\mathcal E}\prec U_\sigma.\end{equation}
Therefore, $V^*|_{\mathcal K\ominus\mathcal E}$ is of class $C_{\cdot 1}$. Consequently, 
$V^*|_{\mathcal K\ominus\mathcal E}$ is unitary. Thus,  
$\mathcal K\ominus\mathcal E\in\operatorname{Lat}V$, and 
$V|_{\mathcal K\ominus\mathcal E}\cong U_\sigma$ by \eqref{congu} and  {\cite[Proposition II.3.4]{sfbk}}.
Without loss of generality, $\mathcal K\ominus\mathcal E =  L^2(\sigma)$ and 
$V|_{\mathcal K\ominus\mathcal E}= U_\sigma$. We have $\ker XY|_{ L^2(\sigma)}=\{0\}$. 
Therefore, there exists $g\in L^\infty(\sigma)$ such that 
\begin{equation}\label{kerxxyy} g\neq 0 \ \ \text{ a.e. on } \sigma \ \text{ and } XYh=gh \ \text{ for every }
h\in L^2(\sigma).\end{equation}

For $0<\varepsilon_0,\delta<1$ set 
$$\delta_1=\frac{\delta-\varepsilon_0\|X\|}{1+\varepsilon_0} \ \ \text{   and } \ \   
\gamma_1=\Bigl(1-\frac{\varepsilon_0^2\|X\|^2}{(1-\varepsilon_0)^2\delta_1^2}\Bigr)^{1/2}.$$
Then $\gamma_1\to 1$ and  $\delta_1\to 1$  when $\delta\to 1$ and $\varepsilon_0\to 0$. 
Choose $0<\varepsilon_0,\delta<1$ such that $0<\varepsilon_0\leq\varepsilon_{00}$, $\gamma_0<\gamma_1$ 
 and $\delta_0\leq\delta_1$. 
We claim that $m(\tau_{\varepsilon_0,\delta_0,\gamma_0})=0$.

Set $\tau=\tau_{\varepsilon_0,\delta_0,\gamma_0}$. Let us assume that  $m(\tau)>0$. 
Set $\mathcal M=\operatorname{clos}Y L^2(\tau)$. 
By \eqref{kerxxyy}, $X|_{\mathcal M}\not\equiv\mathbb O$. By  Lemma \ref{asymp1}, $\|X|_{\mathcal M}\|\geq 1$. Therefore, there exists $u\in\mathcal M$ such that $\|u\|=1$ and 
$\|Xu \|\geq \delta$. We have 
$$Xu\in\operatorname{clos}X\mathcal M=\operatorname{clos}XYL^2(\tau)= L^2(\tau).$$
Since $\{\varphi_n(T)x\}_{n\in\mathbb N}$ is dense in $\mathcal H$, 
there exists $n\in\mathbb N$ such that $$\|\varphi_n(T)x-u\|\leq\varepsilon_0.$$ 
Therefore, $n\in \mathbb N_{\varepsilon_0,\delta_0}$.
By Proposition \ref{proptau1} applied to  $x$ and $\varphi_n$ with $\varepsilon_0$ and $\delta$, taking into account that $\gamma_0<\gamma_1$, we obtain $\tau\not\subset\sigma_{n, \gamma_0\delta_1}$. 
By the definition of $\tau$, we have  $\tau\subset\sigma_{n,\gamma_0\delta_0}\subset\sigma_{n, \gamma_0\delta_1}$, 
a contradiction.
\end{proof}

 The following lemma is a version of  {\cite[Theorem 3]{ker89}}.

  \begin{lemma}\label{ker89thm3} Suppose that $\mathop{\text{\rm Lim}}$ is a Banach limit, $T$ is a power bounded operator, 
$\mathcal M\in\operatorname{Lat}T$, $X_*$ and $Z$ are canonical intertwining mappings 
for $T^*$ and $(T|_{\mathcal M})^*$ from their unitary asymptotes, respectively, constructed using 
$\mathop{\text{\rm Lim}}$.  
Set $$\mathcal L=\bigvee_{n\geq 0}\bigl((T^*)^{(a)}\bigr)^{-n}X_*\mathcal M^\perp.$$
Then there exists a transformation $B$ such that $$Z=B P_{\mathcal L^\perp}X_*|_{\mathcal M} \ \ \text{ and } \ \  
\|B\|\leq\sqrt 2\|X_*\|.$$
\end{lemma}
\begin{proof} By the constructions of  $X_*$ and $Z$, 
 $$\|Zx\|^2=\mathop{\text{\rm Lim}}_n\|\bigl((T|_{\mathcal M})^*\bigr)^nx\|^2=
\mathop{\text{\rm Lim}}_n\| P_{\mathcal M}(T^*)^n x\|^2
\leq\mathop{\text{\rm Lim}}_n\|(T^*)^n x\|^2 \!=\! \|X_*x\|^2$$
for every $x\in\mathcal M$. Therefore, $\|Z\|\leq \|X_*\|$.

Following the proof of  {\cite[Theorem 3(b)]{ker89}} (applied to $T^*$) we obtain a transformation
$\begin{pmatrix}B_1 & B_2 \\ \mathbb O & B\end{pmatrix}$ such that 
\begin{align*}\begin{pmatrix}B_1 & B_2 \\ \mathbb O & B\end{pmatrix}
\begin{pmatrix}X_*|_{\mathcal M^\perp} & P_{\mathcal L}X_*|_{\mathcal M} \\ 
\mathbb O & P_{\mathcal L^\perp}X_*|_{\mathcal M}\end{pmatrix} & =
\begin{pmatrix}X_*|_{\mathcal M^\perp} & P_{\mathcal L}X_*|_{\mathcal M} \\ 
\mathbb O & Z\end{pmatrix} \\ \text{ and } 
\Bigl\|\begin{pmatrix}B_1 & B_2 \\ \mathbb O & B\end{pmatrix}\Bigr\| & \leq
\Bigl\|\begin{pmatrix}X_*|_{\mathcal M^\perp} & P_{\mathcal L}X_*|_{\mathcal M} \\ 
\mathbb O & Z\end{pmatrix} \Bigr\|\leq\sqrt 2\|X_*\|.\end{align*}
Therefore, $B P_{\mathcal L^\perp}X_*|_{\mathcal M}=Z$ and 
$$\|B\|\leq\Bigl\|\begin{pmatrix}B_1 & B_2 \\ \mathbb O & B\end{pmatrix}\Bigr\|\leq\sqrt 2\|X_*\|.\qedhere$$
\end{proof}

\begin{proposition}\label{proptt0tt} Set $\chi(\zeta)=\zeta$ and $\text{\bf 1}(\zeta)=1$ $(\zeta\in\mathbb T)$. 
Set $$H^2_-=L^2(\mathbb T)\ominus H^2 \ \ \text{ and } \ \  S_*=P_{H^2_-}U_{\mathbb T}|_{ H^2_-}.$$  
Suppose that  $\mathcal N$ and $\mathcal K$ are Hilbert spaces, $\mathop{\text{\rm Lim}}$ is a Banach limit, 
$$T_0\in\mathcal L(\mathcal N)\ \text{  is power bounded,}$$ $V_0\in\mathcal L(\mathcal K)$, $Z_0\in\mathcal L(\mathcal N,\mathcal K)$, 
$(Z_0,V_0)$ is the unitary asymptote of $T_0^*$ constructed using $\mathop{\text{\rm Lim}}$, $Y_0\in\mathcal L(H^2,\mathcal N)$, 
and $Y_0 S = T_0Y_0$. Set $$T=\begin{pmatrix} T_0 & (\cdot,\overline\chi)Y_0\text{\bf 1}\\ \mathbb O & S_*\end{pmatrix}
 \ \ \text{ and } \ \ X_*=\begin{pmatrix}Z_0 & \mathbb O \\ Y_0^* &  \mathbb O \\  \mathbb O & I_{H^2_-}\end{pmatrix}.$$
Then 
\begin{equation}\label{yy0oplus}(Y_0\oplus I_{H^2_-})U_{\mathbb T}=T(Y_0\oplus I_{H^2_-}),\end{equation} 
$T$ is power bounded, and $(X_*, V_0\oplus U_{\mathbb T}^*)$ 
is the unitary asymptote of $T^*$ constructed using $\mathop{\text{\rm Lim}}$. Furthermore, if $T_0$ is polynomially bounded, then $T$ is 
polynomially bounded, and if $T_0$ is a.c., then $T$ is a.c..

Moreover, let $(X_0,U_{\mathbb T})$ be the unitary asymptote of $T_0$ constructed using some  Banach limit. 
Since  $(X_0Y_0)S=U_{\mathbb T}(X_0Y_0)$, there exists $g_1\in L^\infty(\mathbb T)$ such that 
$X_0Y_0h=g_1h$ for every $h\in H^2$. Let $$X\in\mathcal L( \mathcal N\oplus  H^2_-, L^2(\mathbb T))$$ be defined by the formula
$$ X(x\oplus h)=X_0 x + g_1 h \ \ \ \ \ (x\in \mathcal N, \ h\in  H^2_-).$$ 
Then $(X,U_{\mathbb T})$ is the unitary asymptote of $T$ constructed using the same  Banach limit as $X_0$.
\end{proposition}

\begin{proof} The equality \eqref{yy0oplus} easy follows from the definition of $T$. Let $p$ be an (analytic) polynomial. 
It easy follows from \eqref{yy0oplus} that 
\begin{equation}\label{ptt} P_{\mathcal N}p(T)|_{H^2_-}=Y_0P_{H^2}p(U_{\mathbb T})|_{H^2_-}.\end{equation}  
The conclusions on power boundedness  and polynomial boundedness of $T$ follow from \eqref{ptt}. Since 
the right part of \eqref{ptt} is defined for every $p\in H^\infty$, the conclusion 
on a.c. polynomial boundedness of $T$ follows from \eqref{ptt} and  {\cite[Theorem 23]{ker16}}.

Let $x_1$, $x_2\in\mathcal N$,  let $h_1$, $h_2\in H^2_-$, and let $n\in\mathbb N$. It follows from  \eqref{ptt} that
\begin{align*} ((T^*)^n & (x_1\oplus h_1), (T^*)^n(x_2\oplus h_2))\\ & 
= (h_1,h_2)+(P_{H^2_-}\overline\chi^n Y_0^*x_1,P_{H^2_-}\overline\chi^n Y_0^*x_2)
+((T_0^*)^nx_1,(T_0^*)^nx_2).\end{align*}
 We have  
$$\lim_n(P_{H^2_-}\overline\chi^n Y_0^*x_1,P_{H^2_-}\overline\chi^n Y_0^*x_2)=(Y_0^*x_1,Y_0^*x_2).$$
 By the construction of $Z_0$,  $$(Z_0x_1,Z_0x_2)=\mathop{\text{\rm Lim}}_n((T_0^*)^nx_1,(T_0^*)^nx_2),$$
and by the construction of $X_*$, 
 $$(X_*(x_1\oplus h_1),X_*(x_2\oplus h_2))=\mathop{\text{\rm Lim}}_n((T^*)^n(x_1\oplus h_1), (T^*)^n(x_2\oplus h_2)).$$
Therefore, $$(X_*(x_1\oplus h_1),X_*(x_2\oplus h_2))=(h_1,h_2)+(Y_0^*x_1,Y_0^*x_2)+ (Z_0x_1,Z_0x_2).$$
From the latest equality and  {\cite[Theorem 3]{ker89}} we obtain the conclusion on the unitary asymptote of $T^*$.

Suppose that $(X_0,U_{\mathbb T})$ is the unitary asymptote of $T_0$ constructed using some  Banach limit, and $X$ is defined as above.
Clearly, $XT=U_{\mathbb T}X$. 
Denote by $X_T$ the canonical intertwining mapping for $T$ and $T^{(a)}$ constructed using the same Banach limit as $X_0$. 
By  {\cite[Theorem 3]{ker89}}, $T^{(a)}\cong (T_0)^{(a)}$ and $X_0=X_T|_{\mathcal N}$. 
Therefore, $T^{(a)}\cong U_{\mathbb T}$, and it is easy to see from the definitions of $T$ and $X$ 
and the intertwining properties of $X$ and $X_T$ that $X_T=X$.
\end{proof}

\section{Proof of Theorem \ref{thmmain}}

For convenience of applying  Proposition  \ref{proptt0tt}, we denote the operator from Theorem \ref{thmmain} 
by $T_0$. 

Suppose that  $\mathcal N$ is a Hilbert space, $T_0\in\mathcal L(\mathcal N)$ is a  power bounded operator of class $C_{1\cdot}$,
and $(T_0)^{(a)}\cong U_{\mathbb T}$. Let 
$$T_0=\begin{pmatrix}T_{11} & * \\ \mathbb O & T_{\cdot 0}\end{pmatrix}$$
be the triangulation of $T_0$ of the form $\begin{pmatrix}C_{\cdot 1} & * \\ \mathbb O & C_{\cdot 0}\end{pmatrix}$, 
see Introduction and \cite{ker89}. 
Since $T_0$ is  of class $C_{1\cdot}$, we have $T_{11}$ is of class $C_{11}$.   
By  {\cite[Theorem 3]{ker89}} and  {\cite[Propositions II.3.4 and II.5.3]{sfbk}}, 
$T_{11}\sim (T_{11})^{(a)}$, $(T_0^*)^{(a)}\cong ((T_{11})^{(a)})^*$, 
and there exists a measurable set  $\tau\subset\mathbb T$ such that $(T_{11})^{(a)}\cong U_\tau$. 

{\bf Two cases are possible.}  \emph{ First case:} $m(\tau)=1$.  Then $T_{11}\sim U_{\mathbb T}$. Suppose that $T_0$ is  polynomially bounded. We can consider $T_{11}$ instead of $T_0$. The proof becomes simpler. Proposition \ref{proptt0tt} is not used, because we need not to  represent $T_{11}$ as a restriction of some operator of class $C_{\cdot 1}$ on its invariant subspace $\mathcal N$. 
Lemma \ref{lemmin} is not used, because we need not to request that the needed shift-type invariant subspace is contained in the given subspace  $\mathcal N$. The  detailed proof can be deduced from the proof of second case and is not given here.
 
\emph{ Second  case:} $ m(\tau)<1$. 

Suppose that $T_0$ is a cyclic a.c. polynomially bounded operator. By Corollary \ref{tha},  
there exists a quasiaffinity  $Y_0\in\mathcal L(H^2,\mathcal N)$ such that $Y_0 S = T_0Y_0$. 
Denote by $X_0$ the canonical intertwining mapping for $T_0$ and  $U_{\mathbb T}$ 
constructed using a Banach limit. We have 
\begin{equation}\label{xx0}X_0T_0=U_{\mathbb T}X_0 \ \ \text{ and } \ \ \bigvee_{n\geq 0}U_{\mathbb T}^{-n}X_0\mathcal N=L^2(\mathbb T).\end{equation} 
Set $\mathcal H=\mathcal N\oplus H^2_-$. Let $T\in\mathcal L(\mathcal H)$, 
$X\in\mathcal L(\mathcal H, L^2(\mathbb T))$, $g_1\in L^\infty(\mathbb T)$ be    from 
Proposition \ref{proptt0tt} applied to $T_0$ and $Y_0$. Since $Y_0$ is a quasiaffinity, it follows from the second equality in \eqref{xx0} that 
   $g_1\neq 0$ a.e. on $\mathbb T$, and  it follows from \eqref{yy0oplus} that $T$ is of class $C_{\cdot 1}$. 

Let $X_*\in\mathcal L(\mathcal H, L^2(\tau)\oplus L^2(\mathbb T))$ be the canonical intertwining mapping for $T^*$ 
constructed in Proposition \ref{proptt0tt} with $V_0=U_\tau^*$. Then  
\begin{equation}\label{xxstarnn} X_*^*(L^2(\tau)\oplus H^2)\subset\mathcal N.\end{equation}
Furthermore,  there exists $g_2\in L^\infty(\tau)$ such that $$(XX_*^*)(h_2\oplus h_1) = g_2h_2+g_1h_1 \ \ \text{ for every } 
h_2\in L^2(\tau)  \text{ and } h_1 \in L^2(\mathbb T).$$
Set $$ \theta_1=\frac{g_1}{(|g_1|^2+|g_2|^2)^{1/2}} \ \ \text{ and } \ \ \theta_2=\frac{g_2}{(|g_1|^2+|g_2|^2)^{1/2}},$$
and set $$ \mathcal E_0=\{ \theta_1 h\oplus(-\theta_2h)\ :\ h\in L^2(\tau)\}.$$
Clearly, $\mathcal E_0=\ker XX_*^*$ and  $\mathcal E_0$ is a reducing subspace for $U_\tau\oplus U_{\mathbb T}$,  
that is, 
$$\mathcal E_0\in\operatorname{Lat}(U_\tau\oplus U_{\mathbb T}) \text{ and } 
(L^2(\tau)\oplus L^2(\mathbb T))\ominus \mathcal E_0\in\operatorname{Lat}(U_\tau\oplus U_{\mathbb T}).$$
 Define 
$$J\in\mathcal L\bigl(L^2(\mathbb T),(L^2(\tau)\oplus L^2(\mathbb T))\ominus \mathcal E_0\bigr)$$
by the formula 
$$Jh=\overline\theta_2 h|_\tau\oplus (\overline\theta_1 h|_\tau + h|_{\mathbb T\setminus\tau}) \ \ \ (h\in L^2(\mathbb T)).$$
Set \begin{equation}\label{gnew} g=\begin{cases}(|g_1|^2+|g_2|^2)^{1/2} & \text{ on } \tau, \\
g_1 & \text{ on } \mathbb T\setminus\tau.\end{cases}\end{equation}
 Then $g\neq 0$ a.e. on $\mathbb T$, $J$ is unitary, $JU_{\mathbb T}=(U_\tau\oplus U_{\mathbb T})|_{(L^2(\tau)\oplus L^2(\mathbb T))\ominus \mathcal E_0}J$, 
\begin{equation}\label{xxxxstarjj}  XX_*^*J h=gh \ \text{ for every } h\in L^2(\mathbb T) \ \text{ and } \ L^2(\tau)\oplus L^2(\mathbb T)= JL^2(\mathbb T)\oplus \mathcal E_0.\end{equation}

Take $\sigma\subset \mathbb T$ such that  
\begin{equation}\label{vyborsigma}K C_{{\rm pol}, T} \sqrt 2\|X_*\|\mathop{\text{\rm ess\,sup}}_{\sigma}|g| < 1, \ \ 
\mathop{\text{\rm ess\,inf}}_{\mathbb T\setminus\sigma}|g| > 0 \ \text{ and } m(\tau\cup\sigma)<1.\end{equation}

Let $h_0\in H^2$ be such that 
\begin{equation}\label{choiseh0}|h_0|\asymp 1 \ \text{  a.e. on } \mathbb T.\end{equation}
 There exists $\theta_0\in L^\infty(\tau)$ such that $|\theta_0|=1$ a.e. on $\tau$ and 
\begin{equation}\label{real}\operatorname{Re} \theta_1\overline\theta_2\overline\theta_0 =0  \ \text{ a.e. on  } \tau.\end{equation}
Set 
\begin{equation}\label{h01new}
h_{01}=(\theta_2\theta_0+\theta_1)h_0\chi_\tau + h_0\chi_{\mathbb T\setminus\tau}\  \ \ \text{ and }\ \  \ 
 h_{02}= (-\overline\theta_2+\overline\theta_1\theta_0)h_0|_\tau. \end{equation}
By \eqref{real}, 
\begin{equation}\label{mod}|h_{01}|=|h_0| \ \text{ a.e. on } \mathbb T.\end{equation}

Set \begin{equation}\label{ffsigma} \begin{aligned} 
 F_{\tau,\sigma}&=\overline\theta_2h_{01}|_{\tau\cap\sigma}+ \theta_1 h_{02}, \\
 F_{\mathbb T,\sigma}&=\overline\theta_1h_{01}|_{\tau\cap\sigma}+h_{01}|_{\sigma\setminus\tau}
-\theta_2 h_{02}, \\ 
F_{\tau,\mathbb T\setminus\sigma}&=\overline\theta_2h_{01}|_{\tau\setminus\sigma}, \\
F_{\mathbb T,\mathbb T\setminus\sigma}&=\overline\theta_1h_{01}|_{\tau\setminus\sigma}+h_{01}|_{(\mathbb T\setminus\tau)\setminus\sigma}.
\end{aligned}\end{equation}
Then
\begin{equation}\label{ffjjsigma} F_{\tau,\sigma}\oplus F_{\mathbb T,\sigma} = J h_{01}|_\sigma + \theta_1h_{02}\oplus(-\theta_2h_{02}), 
\ \ \ \ \theta_1h_{02}\oplus(-\theta_2h_{02})\in\mathcal E_0,\end{equation}
\begin{equation}\label{ffjjsetminussigma}
F_{\tau,\mathbb T\setminus\sigma}\oplus F_{\mathbb T,\mathbb T\setminus\sigma} = J h_{01}|_{\mathbb T\setminus\sigma},\end{equation}
and \begin{equation}\label{ffh0}
F_{\tau,\sigma}\oplus F_{\mathbb T,\sigma} + F_{\tau,\mathbb T\setminus\sigma}\oplus F_{\mathbb T,\mathbb T\setminus\sigma}=
\theta_0 h_0|_\tau\oplus h_0\in L^2(\tau)\oplus H^2.\end{equation}
Put \begin{equation}\label{x0ppsigma} 
 x_0  =X_*^*(F_{\tau,\sigma}\oplus F_{\mathbb T,\sigma}).
\end{equation}
It follows from \eqref{ffjjsigma} and \eqref{xxxxstarjj} that 
 \begin{equation}\label{xxpp} X x_0=gh_{01}|_\sigma.
 \end{equation}

Set $$\mathcal M_0=\bigvee_{k\geq 0}T^k x_0 \ \ \text{ and } \ \  
\mathcal K_0= \bigvee_{k\geq 0}(U_\tau\oplus U_{\mathbb T})^k(F_{\tau,\sigma}\oplus F_{\mathbb T,\sigma}). $$
By \eqref{ffsigma}, $F_{\tau,\sigma}\oplus F_{\mathbb T,\sigma}\in L^2(\tau)\oplus  L^2(\tau\cup\sigma)$. 
It follows from third inequality in \eqref{vyborsigma} that $(U_\tau\oplus U_{\mathbb T})|_{\mathcal K_0}$ is unitary. 
It follows from \eqref{choiseh0}, \eqref{mod},   \eqref{xxpp}, and the equality $XT=U_{\mathbb T}X$ that 
$$\operatorname{clos}X\mathcal M_0 = L^2(\sigma).$$
By  {\cite[Theorem 3]{ker89}}, $(X|_{\mathcal M_0}, U_\sigma)$ is the unitary asymptote of $T|_{\mathcal M_0}$ constructed using the same Banach limit as $X$. 

Set \begin{equation}\label{ll1def}\mathcal L=\bigvee_{n\geq 0}(U_\tau^*\oplus U_{\mathbb T}^*)^{-n}X_*\mathcal M_0^\perp.\end{equation} 
Denote by $Z$ the canonical intertwining mapping for $(T|_{\mathcal M_0})^*$ constructed using the same Banach limit as  $X_*$.
Let $B$ be the transformation from Lemma \ref{ker89thm3} applied to $T$ and $\mathcal M_0$. 
We have 
$$ X|_{\mathcal M_0}Z^*=XP_{\mathcal M_0}X_*^*P_{\mathcal L^\perp}B^*.$$
Therefore, 
$$\|X|_{\mathcal M_0}Z^*\|\leq\sqrt 2\|X_*\|\|XP_{\mathcal M_0}X_*^*|_{\mathcal L^\perp}\|.$$
It follows from \eqref{ll1def} that $ X_*^*\mathcal L^\perp\subset\mathcal M_0$.
Therefore, \begin{equation}\label{xxstarll1}XP_{\mathcal M_0}X_*^*|_{\mathcal L^\perp}=XX_*^*|_{\mathcal L^\perp}\ \ \text{ and } \ \  
 XX_*^*\mathcal L^\perp\subset L^2(\sigma).\end{equation}
 Let $u\in\mathcal L^\perp$. By \eqref{xxxxstarjj},  there exist $h\in L^2(\mathbb T)$ and $v\in \mathcal E_0$ 
such that $u = Jh + v$. By \eqref{xxxxstarjj} and \eqref{xxstarll1},
$$ XX_*^* u =XX_*^* J h=gh\in L^2(\sigma).$$
Therefore, $h|_{\mathbb T\setminus\sigma} =0$ and 
\begin{align*}\| XX_*^* u\| & \leq \mathop{\text{\rm ess\,sup}}_{\sigma}|g|\|h\|_{L^2(\sigma)}\leq
\mathop{\text{\rm ess\,sup}}_{\sigma}|g|(\|h\|_{L^2(\mathbb T)}^2 + \|v\|^2)^{1/2}  \\ &=
\mathop{\text{\rm ess\,sup}}_{\sigma}|g|\|u\|\end{align*}
(because $J$ is unitary). We obtain that 
$$ \|XX_*^*|_{\mathcal L^\perp}\|\leq\mathop{\text{\rm ess\,sup}}_{\sigma}|g|.$$
Consequently, $$\|X|_{\mathcal M_0}Z^*\|\leq\sqrt 2\|X_*\|\mathop{\text{\rm ess\,sup}}_{\sigma}|g|.$$
By \eqref{vyborsigma},  $$KC_{{\rm pol}, T}\|X|_{\mathcal M_0}Z^*\|<1.$$ 

Take $$KC_{{\rm pol}, T}\|X|_{\mathcal M_0}Z^*\|<\delta_0<1 \  \text{ and } \  0<\varepsilon_{00},\gamma_0<1.$$ 
Let $\{\varphi_n\}_{n\in\mathbb N}\subset H^\infty$ 
 be  such that $\{\varphi_n(T) x_0\}_{n\in\mathbb N}$ is dense in $\mathcal M_0$. 
Let $\psi_n$ be constructed by Corollary \ref{tha} applied to $\varphi_n(T) x_0$ for every $n\in\mathbb N$. Set 
\begin{equation}\label{fdef} f=X x_0. \end{equation}
It is easy to see that $T|_{\mathcal M_0}$, $X|_{\mathcal M_0}$, $(U_\tau\oplus U_{\mathbb T})|_{\mathcal K_0}$, 
$X_*^*|_{\mathcal K_0}$, $x_0$, $\{\psi_n\}_{n\in\mathbb N}$ satisfy   Proposition \ref{proptau2}.
Let $0<\varepsilon_0\leq\varepsilon_{00}$ be 
such that $$m(\tau_{\varepsilon_0,\delta_0,\gamma_0})=0,$$
where $\tau_{\varepsilon_0,\delta_0,\gamma_0}$ is from  Proposition \ref{proptau2}.   

 For convenience, relabel $\{\varphi_n\}_{n\in\mathbb N_{\varepsilon_0,\delta_0}}$ 
as $\{\varphi_n\}_{n\in\mathbb N}$. 
Then $$\|X\varphi_n(T) x_0\|\geq\delta_0 \|\varphi_n(T) x_0\|\ \ \text{  for every } n\in\mathbb N.$$ 
By Lemma \ref{lemast}  applied to $T|_{\mathcal M_0}$ and the choice of $\delta_0$ we have 
\begin{equation}\label{lognew}\int_{\mathbb T}\log|\psi_n|\text{\rm d}m>-\infty\ \ \text{  for every } n\in\mathbb N.\end{equation} 
We accept that $\psi_n$ are outer functions (and $\psi_n\in H^2$), see Remark \ref{h2}.
Set $\sigma_n=\sigma_{n,\gamma_0\delta_0}$, $n\in\mathbb N$. 

By \eqref{fdef} and Lemma \ref{lemfpsi} applied to $\varphi_n(T) x_0$, 
\begin{equation}\label{phifleq}\begin{aligned}
|\varphi_n f| & \leq\|X\|KC_{{\rm pol}, T}\|\varphi_n(T) x_0\||\psi_n| \\ &
\leq(1+\varepsilon_0)\|X\|KC_{{\rm pol}, T}|\psi_n| \ \ \text{ a.e. on }\sigma \text{ for every } n\in\mathbb N.\end{aligned}\end{equation}
By Proposition \ref{proptau2}, 
$$|\varphi_nf|>\gamma_0\delta_0\|\varphi_n(T)x\||\psi_n|
\geq (1-\varepsilon_0)\gamma_0\delta_0|\psi_n| \ \ \text{  a.e. on }\ \sigma\setminus\sigma_n $$
 for every $n\in\mathbb N$  
and 
$$ m(\bigcap_{n\in\mathbb N}\sigma_n)=0. $$
Set 
$$ \tau_1=\sigma\setminus\sigma_1, 
\ \  \tau_n=(\sigma\setminus\sigma_n)\setminus\cup_{k=1}^{n-1}\tau_k, \ \ n\geq 2.$$
Then 
$$ \bigcup_{n\in\mathbb N}\tau_n=\sigma, \ \ \ \tau_n\cap\tau_k=\emptyset, \ \text{ if } \ n\neq k, \ n,k\geq 1,$$
and 
\begin{equation}\label{fphigeq} |\varphi_nf|>(1-\varepsilon_0)\gamma_0\delta_0|\psi_n| 
\ \text{ a.e. on } \tau_n, \ n\in\mathbb N.
\end{equation}

By \eqref{x0ppsigma},
$$ \varphi(T)\varphi_n(T) x_0  =\varphi(T)\varphi_n(T)X_*^*(F_{\tau,\sigma}\oplus F_{\mathbb T,\sigma})
=X_*^*(\varphi\varphi_nF_{\tau,\sigma}\oplus\varphi\varphi_n F_{\mathbb T,\sigma}).$$
Therefore, 
\begin{align*} \|\varphi(T)\varphi_n(T) x_0\|^2 
& \leq \|X_*^*\|^2(\|\varphi\varphi_nF_{\tau,\sigma}\|^2 + \|\varphi\varphi_n F_{\mathbb T,\sigma}\|^2)\\&=
\|X_*^*\|^2\Bigl(\int_\tau|\varphi\varphi_nF_{\tau,\sigma}|^2 \text{\rm d}m+ 
\int_{\mathbb T}|\varphi\varphi_n F_{\mathbb T,\sigma}|^2\text{\rm d}m\Bigr) \\ & = 
\|X_*^*\|^2\int_{\mathbb T}|\varphi\varphi_n|^2(\chi_\tau|F_{\tau,\sigma}|^2 + |F_{\mathbb T,\sigma}|^2)\text{\rm d}m \\ &
\leq \|X_*^*\|^2\mathop{\text{\rm ess\,sup}}_{\mathbb T}(\chi_\tau|F_{\tau,\sigma}|^2 + 
|F_{\mathbb T,\sigma}|^2)\int_{\mathbb T}|\varphi\varphi_n|^2\text{\rm d}m.
\end{align*}
Easy calculation show that 
$$\mathop{\text{\rm ess\,sup}}_{\mathbb T}(\chi_\tau|F_{\tau,\sigma}|^2 + 
|F_{\mathbb T,\sigma}|^2)=\max(2\mathop{\text{\rm ess\,sup}}_{\tau\cap\sigma}|h_0|^2, 
\mathop{\text{\rm ess\,sup}}_{(\tau\setminus\sigma)\cup(\sigma\setminus\tau)}|h_0|^2).$$
Set $$C_1=\frac{\|X_*^*\|\bigl(\max(2\mathop{\text{\rm ess\,sup}}_{\tau\cap\sigma}|h_0|^2, 
\mathop{\text{\rm ess\,sup}}_{(\tau\setminus\sigma)\cup(\sigma\setminus\tau)}|h_0|^2)\bigr)^{1/2}}{KC_{{\rm pol}, T}(1+\varepsilon_0)}.$$
Then 
$$\|\varphi(T)\varphi_n(T) x_0\|\leq KC_{{\rm pol}, T}(1+\varepsilon_0)
\Bigl(\int_{\mathbb T}|\varphi|^2C_1^2|\varphi_n|^2\text{\rm d}m\Bigr)^{1/2}$$ 
for every $\varphi\in H^\infty$ and every $n\in\mathbb N$.

Take $0<c_1<1$. By Lemma \ref{lemmin}, there exist inner functions $\omega_n\in H^\infty$, $n\in\mathbb N$, such that 
\begin{align*}\|\varphi(T) & \omega_n(T) \varphi_n(T) x_0\| \\&
\leq KC_{{\rm pol}, T}(1+\varepsilon_0)\frac{(2(1+c_1^2))^{1/2}}{1-c_1}
\Bigl(\int_{\mathbb T}|\varphi|^2\min(|\psi_n|, C_1|\varphi_n|)^2\text{\rm d}m\Bigr)^{1/2}\end{align*}
for every $\varphi\in H^\infty$ and every $n\in\mathbb N$.

There exist outer functions $\psi_{1n}\in H^2$ such that  
 \begin{equation}\label{psi1ndef} |\psi_{1n}|=\begin{cases}|\psi_n| & \text{ on } \sigma, \\
\min(|\psi_n|, C_1|\varphi_n|) & \text{ on }  \mathbb T\setminus\sigma.\end{cases}\end{equation}
Since $\min(|\psi_n|, C_1|\varphi_n|)\leq |\psi_{1n}|$ a.e. on $\mathbb T$, we have
\begin{equation}\label{w1n}\begin{aligned}\|\varphi(T)\omega_n(T)\varphi_n(T) x_0\| 
\leq &  KC_{{\rm pol}, T}(1+\varepsilon_0)\frac{(2(1+c_1^2))^{1/2}}{1-c_1}
\Bigl(\int_{\mathbb T}|\varphi|^2|\psi_{1n}|^2\text{\rm d}m\Bigr)^{1/2}\\  & \text{ for every } \ \varphi\in H^\infty
\text{ and  every } \ n\in\mathbb N.\end{aligned}\end{equation}
It follows from \eqref{psi1ndef} that  \begin{equation}\label{phipsin}\Bigl|\frac{\varphi_n}{\psi_{1n}}\Bigr|\geq\frac{1}{C_1} 
\ \ \ \text{ on }  \mathbb T\setminus\sigma \text{ for every } n\in\mathbb N.\end{equation} 

Let $\{\xi_n\}_{n\in\mathbb N}$ be a sequence of positive numbers such that 
\begin{equation}\label{xi2} \xi:=\sum_{n\in\mathbb N}\xi_n^2 <\infty, \ \ \ \ \ \ \sum_{n\geq 2}\xi_n^2 <1
\end{equation} and 
\begin{equation}\label{ximinus}
 (1+\varepsilon_0)\|X\|KC_{{\rm pol}, T}\xi <(1-\varepsilon_0)\gamma_0\delta_0.\end{equation}

For $n\in\mathbb N$, define  outer functions $\eta_n\in H^\infty$ as follows:
\begin{equation}\label{etadef}
|\eta_1|  =\begin{cases}\Bigl|\frac{\psi_{11}}{\varphi_1}\Bigr|^{1/2}  & \text{ on } \mathbb T\setminus\sigma, \\
1  & \text{ on } \tau_1, \\ \xi_1 & \text{ on } \sigma\setminus\tau_ 1,\end{cases} \ \ \ \ 
|\eta_n|  =\begin{cases}\xi_n\Bigl|\frac{\psi_{1n}}{\varphi_n}\Bigr|^{1/2}  & \text{ on } \mathbb T\setminus\sigma, \\ 
1  & \text{ on } \tau_n, \\
\xi_n & \text{ on } \sigma\setminus\tau_ n,\end{cases} n\geq 2.
\end{equation}
By  \eqref{phipsin} and \eqref{xi2}, 
$$ \sum_{n\in\mathbb N} |\eta_n(\zeta)|^2\leq \begin{cases} (1+\sum_{n\geq 2}\xi_n^2)C_1 & \text{ for a.e.  }\zeta\in \mathbb T\setminus\sigma, \\
1+\xi  & \text{ for a.e.  }\zeta\in\sigma.\end{cases}$$

For every $t>0$ and every $n\in\mathbb N$ let $\psi_{nt}$ be the outer function such that
 \begin{equation}\label{psintdef} |\psi_{nt}|=\max(|\psi_{1n}|,t) \ \ \text{ a.e. on } \mathbb T. \end{equation}

Put  \begin{equation}\label{kappa}
\kappa_{nt}(\zeta)=\frac{\eta_n(\zeta)}{\psi_{nt}(\zeta)}, \ \ \ \zeta\in\mathbb T, \ n\in\mathbb N, \  t>0.
\end{equation} 
Since $|\psi_{nt}|\geq t$ a.e. on $\mathbb T$, we have $\kappa_{nt}\in H^\infty$ for every $n\in \mathbb N$ and $t>0$. 
Set  
 \begin{equation}\label{xynt}\begin{aligned}
x_{nt} & =(\kappa_{nt}\omega_n\varphi_n)(T)x_0=(\kappa_{nt}\omega_n\varphi_n)(T)X_*^*(F_{\tau,\sigma}\oplus F_{\mathbb T,\sigma}) \\ \text{ and } \ 
y_{nt} &=  X_*^*( \kappa_{nt}\omega_n\varphi_nF_{\tau,\mathbb T\setminus\sigma}
\oplus \kappa_{nt}\omega_n\varphi_nF_{\mathbb T,\mathbb T\setminus\sigma}) \\ &= 
(\kappa_{nt}\omega_n\varphi_n)(T)X_*^*(F_{\tau,\mathbb T\setminus\sigma}\oplus F_{\mathbb T,\mathbb T\setminus\sigma}) \\ 
 &\ \ \ \ \  \ \ \ \ \ \ \text{  for every } n\in\mathbb N\ \text{ and }t>0.\end{aligned}\end{equation}
We have 
$$ x_{nt}+y_{nt}=(\kappa_{nt}\omega_n\varphi_n)(T)
X_*^*(F_{\tau,\sigma}\oplus F_{\mathbb T,\sigma} + F_{\tau,\mathbb T\setminus\sigma}\oplus F_{\mathbb T,\mathbb T\setminus\sigma}).$$
By \eqref{ffh0} and \eqref{xxstarnn},
\begin{equation}\label{xynn} x_{nt} + y_{nt}\in\mathcal N.\end{equation} 

By \eqref{w1n} and \eqref{psintdef}, 
 \begin{equation}\label{kappaphi1}\begin{aligned}
\sum_{n\in\mathbb N}\|\varphi(T)x_{nt}\|^2 & 
=\sum_{n\in\mathbb N}\|(\varphi\kappa_{nt}\omega_n\varphi_n)(T)x_0\|^2  
\\ & 
\leq \Bigl(KC_{{\rm pol}, T}(1+\varepsilon_0)\frac{(2(1+c_1^2))^{1/2}}{1-c_1}\Bigr)^2
 \sum_{n\in\mathbb N}\int_{\mathbb T}|\varphi\kappa_{nt}|^2|\psi_{nt}|^2\text{\rm d}m,  \end{aligned}\end{equation}
and by \eqref{kappa},
\begin{equation}\label{kappaphi2}\begin{aligned}
\sum_{n\in\mathbb N}\int_{\mathbb T}|\varphi\kappa_{nt}|^2|\psi_{nt}|^2\text{\rm d}m & = 
\sum_{n\in\mathbb N}\int_{\mathbb T}|\varphi|^2|\eta_n|^2\text{\rm d}m \\ =
\int_{\mathbb T}\sum_{n\in\mathbb N}|\eta_n|^2|\varphi|^2\text{\rm d}m
& \leq\mathop{\text{\rm ess\,sup}}_{\mathbb T}\sum_{n\in\mathbb N}|\eta_n|^2
\int_{\mathbb T}|\varphi|^2\text{\rm d}m.\end{aligned}\end{equation}
By Theorem \ref{thinfinit}, \eqref{kappaphi1} and \eqref{kappaphi2}, 
\begin{equation}\label{wwtsigma}\begin{aligned}\|\sum_{n\in\mathbb N} & (\eta_n\varphi)(T)x_{nt}\| \\ &\leq
K^2C_{{\rm pol}, T}^2(1+\varepsilon_0)\frac{(2(1+c_1^2))^{1/2}}{1-c_1}
\mathop{\text{\rm ess\,sup}}_{\mathbb T}\sum_{n\in\mathbb N}|\eta_n|^2
\Bigl(\int_{\mathbb T}|\varphi|^2\text{\rm d}m\Bigr)^{1/2} 
 \\ &  \ \ \  \ \ \ \ \ \ \ \ \ \ \ \ \ \ \ \ \ \ \ \ \ \  \  \ \ \  \ \ \      \ \text{  for every } t>0 .\end{aligned}\end{equation}

For $t>0$ put 
\begin{equation}\label{beta} 
\beta_t=\sum_{n\in\mathbb N}\eta_n^2\frac{\omega_n\varphi_n}{\psi_{nt}} \ \text{ on } \mathbb T\setminus\sigma.\end{equation}
 It follows from \eqref{etadef} and \eqref{psintdef} that the series \eqref{beta} converges in $\|\cdot\|_{L^\infty(\mathbb T\setminus\sigma)}$ 
 and
\begin{equation}\label{betanorm}\|\beta_t\|_{L^\infty(\mathbb T\setminus\sigma)} \leq 1+\sum_{n\geq 2}\xi_n^2
\ \ \  \ \text{ for all } t>0.\end{equation}
 For $t>0$ and $\varphi\in H^\infty$ let
\begin{equation}\label{wwtpsi}W_t\varphi = \sum_{n\in\mathbb N}(\eta_n\varphi)(T)x_{nt} + 
 X_*^*(\beta_t\varphi F_{\tau,\mathbb T\setminus\sigma}\oplus \beta_t\varphi F_{\mathbb T,\mathbb T\setminus\sigma}).\end{equation}
It follows from the convergence in $\|\cdot\|_{L^\infty(\mathbb T\setminus\sigma)}$ of the series \eqref{beta} that 
 \begin{equation}\label{yntconv} \begin{aligned} 
X_*^*(\beta_t\varphi F_{\tau,\mathbb T\setminus\sigma}\oplus & \beta_t\varphi F_{\mathbb T,\mathbb T\setminus\sigma}) \\ &
=\sum_{n\in\mathbb N}X_*^*\Bigl(\eta_n^2\frac{\omega_n\varphi_n}{\psi_{nt}} \varphi F_{\tau,\mathbb T\setminus\sigma}\oplus 
\eta_n^2\frac{\omega_n\varphi_n}{\psi_{nt}} \varphi F_{\mathbb T,\mathbb T\setminus\sigma}\Bigr).
\end{aligned}\end{equation}
It follows from the  intertwining property of $X_*^*$ and \eqref{xynt} that 
\begin{equation}\label{yntequality}X_*^*\Bigl(\eta_n^2\frac{\omega_n\varphi_n}{\psi_{nt}} \varphi F_{\tau,\mathbb T\setminus\sigma}\oplus 
\eta_n^2\frac{\omega_n\varphi_n}{\psi_{nt}} \varphi F_{\mathbb T,\mathbb T\setminus\sigma}\Bigr)=
(\varphi\eta_n)(T)y_{nt}.
\end{equation}
It follows from \eqref{wwtpsi}, \eqref{yntconv}, \eqref{yntequality}, and \eqref{xynn} that 
$$ W_t\varphi\in \mathcal N \text{  for every } \varphi\in H^\infty \text{ and } t>0. $$

By \eqref{betanorm} and \eqref{ffsigma}, 
 \begin{equation}\label{wwtsetminussigma}
\begin{aligned}\|X_*^*&(\beta_t\varphi F_{\tau,\mathbb T\setminus\sigma}\oplus 
\beta_t\varphi F_{\mathbb T,\mathbb T\setminus\sigma})\|^2\\ &
\leq\|X_*^*\|^2\bigl(1+\sum_{n\geq 2}\xi_n^2\bigr)^2 
\mathop{\text{\rm ess\,sup}}_{\mathbb T\setminus\sigma}\bigl(\chi_\tau|F_{\tau,\mathbb T\setminus\sigma}|^2 + 
|F_{\mathbb T,\mathbb T\setminus\sigma}|^2\bigr)
\int_{\mathbb T\setminus\sigma}|\varphi|^2\text{\rm d}m.\end{aligned}\end{equation}
By  \eqref{mod} and \eqref{ffsigma}, 
$$\mathop{\text{\rm ess\,sup}}_{\mathbb T\setminus\sigma}\bigl(\chi_\tau|F_{\tau,\mathbb T\setminus\sigma}|^2 + 
|F_{\mathbb T,\mathbb T\setminus\sigma}|^2\bigr)=
\mathop{\text{\rm ess\,sup}}_{\mathbb T\setminus\sigma}|h_0|^2.$$
Set 
\begin{align*}C_2= & K^2C_{{\rm pol}, T}^2(1+\varepsilon_0)\frac{(2(1+c_1^2))^{1/2}}{1-c_1}
\mathop{\text{\rm ess\,sup}}_{\mathbb T}\sum_{n\in\mathbb N}|\eta_n|^2  \\ & \ \ +
\|X_*^*\|\bigl(1+\sum_{n\geq 2}\xi_n^2\bigr) 
\mathop{\text{\rm ess\,sup}}_{\mathbb T\setminus\sigma}|h_0|.\end{align*}
It follows from  \eqref{wwtsigma} and \eqref{wwtsetminussigma} 
that the mapping $W_t$ defined on $H^\infty$   by \eqref{wwtpsi}
is extended onto $H^2$ for every $t>0$. Denoting the extension by the same letter, we obtain 
\begin{equation}\label{wwt}\begin{aligned}& W_t  \in\mathcal L(H^2, \mathcal N) \ \text{ such that } 
 W_tS=T|_{\mathcal N}W_t \\  &\text{ and } \
\|W_t\| \leq C_2   \ \text{ for every }  t>0.\end{aligned}\end{equation}

Put 
\begin{equation}\label{alphanew} 
\alpha_t=\sum_{n\in\mathbb N}\eta_n^2\frac{\omega_n\varphi_n g h_{01}}{\psi_{nt}} \ \ \text{ and } \ \ 
\alpha=\sum_{n\in\mathbb N}\eta_n^2\frac{\omega_n\varphi_n g h_{01}}{\psi_{1n}} \ \ \text{ on } \mathbb T.
\end{equation}
By \eqref{wwtpsi},   \eqref{xynt},
 \eqref{kappa},  \eqref{ffjjsigma},  \eqref{ffjjsetminussigma}, \eqref{xxxxstarjj}, \eqref{beta},
\begin{equation}\label{xwtalpha} XW_th=\alpha_t h \ \  \ \text{ for every }  t>0\ \text{ and } \ h\in H^2.\end{equation}
By  \eqref{etadef}, 
\eqref{phifleq}, \eqref{psi1ndef},  
 \eqref{fdef}, \eqref{xxpp}, \eqref{psintdef}  and 
Lemma \ref{lempsit}, $\alpha_t$ and $\alpha$ satisfy 
 Proposition \ref{propalphasigma}.
By Proposition \ref{propalphasigma}, 
\begin{equation}\label{alphalim} \sup_{t>0}\|\alpha_t\|_{L^\infty(\sigma)}<\infty\ \text{ and } \ \|\alpha_t-\alpha\|_{L^2(\sigma)}\to 0
\ \text{  when } t\to 0.\end{equation}
An analog of \eqref{alphalim} for $\mathbb T\setminus\sigma$ follows from  \eqref{psintdef}, \eqref{etadef}, 
\eqref{vyborsigma}, \eqref{choiseh0}, \eqref{mod} and 
Lemma \ref{lempsit}. Therefore, 
\begin{equation}\label{alphalimtt} \sup_{t>0}\|\alpha_t\|_{L^\infty(\mathbb T)}<\infty\ \text{ and } \ \|\alpha_t-\alpha\|_{L^2(\mathbb T)}\to 0
\ \text{  when } t\to 0.\end{equation}

It follows from \eqref{wwt} that there exist a sequence $\{t_j\}_j$ and $ W\in\mathcal L(H^2, \mathcal N) $ such that 
$t_j\mathop{\to}_j  0$ 
and $W_{t_j}\mathop{\to}_j W$ in the  weak operator topology, $WS=T|_{\mathcal N}W$ and $\|W\|\leq C_2$. It follows from \eqref{xwtalpha} and \eqref{alphalimtt} that 
\begin{equation}\label{xwalpha}  XWh=\alpha h  \ \  \ \text{ for every }\ h\in H^2.\end{equation}

We will show that  \begin{equation}\label{alphainvtt}1/\alpha\in L^\infty(\mathbb T). \end{equation}
We have $\sigma=\cup_{n\in\mathbb N}\tau_n$.  
Fix $n\in\mathbb N$. By \eqref{phifleq}, \eqref{psi1ndef}, \eqref{xi2}, \eqref{etadef}, 
taking into account that $f=g h_{01}$ a.e. on $\sigma$ by \eqref{fdef} and \eqref{xxpp}, we  obtain 
\begin{align*}\Bigl|\sum_{k\in\mathbb N, k\neq n}\eta_k^2\frac{\omega_k\varphi_k f}{\psi_{1k}}\Bigr| & \leq
\sum_{k\in\mathbb N, k\neq n}|\eta_k|^2\frac{|\varphi_kf|}{|\psi_k|}\leq
(1+\varepsilon_0)\|X\|KC_{{\rm pol}, T}\!\!\sum_{k\in\mathbb N, k\neq n}\xi_k^2 \\ &
\leq(1+\varepsilon_0)\|X\|KC_{{\rm pol}, T}\xi
\ \ \text{ a.e. on }\tau_n.\end{align*}
From the latest estimate, \eqref{etadef} and \eqref{fphigeq} we conclude that 
\begin{align*}|\alpha| &\geq \frac{|\varphi_nf|}{|\psi_n|}-
\Bigl|\sum_{k\in\mathbb N, k\neq n}\eta_k^2\frac{\varphi_k f}{\psi_k}\Bigr| \\ &
\geq (1-\varepsilon_0)\gamma_0\delta_0 - (1+\varepsilon_0)\|X\|KC_{{\rm pol}, T}\xi
\ \text{ a.e. on } \tau_n.\end{align*}
It follows from \eqref{ximinus} that 
\begin{equation}\label{alphainvsigma}1/\alpha\in L^\infty(\sigma). \end{equation}
It follows from \eqref{etadef} that 
$$\Bigl|\sum_{n\in\mathbb N}\eta_n^2\frac{\omega_n\varphi_n}{\psi_{1n}} \Bigr| 
\geq \Bigl|\eta_1^2\frac{\omega_1\varphi_1}{\psi_{11}}\Bigr| - 
\sum_{n\geq 2}\Bigl|\eta_n^2\frac{\omega_n\varphi_n }{\psi_{1n}}\Bigr| = 
1-\sum_{n\geq 2}\xi_n^2\ \ \ \text{ on } \ \mathbb T\setminus\sigma.$$
It follows from \eqref{xi2}, \eqref{choiseh0}, \eqref{mod} and the second inequality in \eqref{vyborsigma}  
that \begin{equation}\label{alphainvsmsigma}1/\alpha\in L^\infty(\mathbb T\setminus\sigma). \end{equation}
Now \eqref{alphainvtt} follows from \eqref{alphainvsigma} and \eqref{alphainvsmsigma}.

Define $ A\in\mathcal L( L^2(\mathbb T))$ by the formula $ Ah=h/\alpha $, $ h\in L^2(\mathbb T)$.
Set $$Z = AX.$$ It follows from \eqref{xwalpha} that 
$$ ZWh=h \ \ \ \ \ \text{for every } h\in H^2.$$ Set $$\mathcal M=W H^2.$$ Then $\mathcal M\subset\mathcal N$ and
 $T|_{\mathcal M}\approx S$. The conclusion of Theorem \ref{thmmain} for $T_0$ follows from the relation $T_0=T|_{\mathcal N}$.

\section{Proof of Corollary \ref{cormain}}

\begin{lemma}\label{lemquasi} Suppose that $T\in\mathcal L(\mathcal H)$ is an  a.c. polynomially bounded operator, 
$X\in\mathcal L(\mathcal H, H^2)$, $Y\in\mathcal L(H^2,\mathcal H)$, $XT=SX$, $YS=TY$, 
$$\operatorname{clos}X\mathcal H=H^2 \ \ \text{ and } \ \ \operatorname{clos}YH^2=\mathcal H.$$
Then $\ker X=\{0\}$ and $\ker Y=\{0\}$.\end{lemma}
\begin{proof} Since $XYS=SXY$ and $\operatorname{clos}XYH^2=H^2$, 
there exists an outer function $g\in H^\infty$ such that $XY=g(S)$. 
Therefore, $\ker Y=\{0\}$. Since 
$YXY=Yg(S)=g(T)Y$ and $\operatorname{clos}YH^2=\mathcal H$, we obtain that $YX=g(T)$. 
By \cite{mlak}, $\ker g(T)=\{0\}$. Therefore, $\ker X=\{0\}$. 
\end{proof}

 \begin{lemma} \label{lemcycss} Suppose that $T$ is a cyclic  a.c. polynomially bounded operator such that 
 $T\buildrel d \over\prec S$. 
Then $T\sim S$. 
\end{lemma}
\begin{proof} Since $T$ is cyclic, 
it follows from Corollary \ref{tha} that $S\buildrel d \over\prec T$. By Lemma \ref{lemquasi}, $T\sim S$. \end{proof}

 \begin{lemma} \label{lemcyc} Suppose that $T$ is a cyclic  a.c. polynomially bounded operator such that 
 $T\buildrel d \over\prec U_{\mathbb T}$. Then $T^{(a)}\cong U_{\mathbb T}$.\end{lemma}
\begin{proof}
 Since $T$ is an  a.c. polynomially bounded operator, $T^{(a)}$ is a.c.   by  {\cite[Proposition 15]{ker16}}.  
By   {\cite[Theorem 2]{ker89}}, $T^{(a)}$ contains $U_{\mathbb T}$ as an orthogonal summand. Since 
$T^{(a)}$ is a.c. and cyclic, we conclude that $T^{(a)}\cong U_{\mathbb T}$.
  \end{proof}

 \begin{lemma}\label{proptrian} Suppose that $T\in\mathcal L(\mathcal H)$ is an a.c. polynomially bounded operator, 
 $$T=\begin{pmatrix}T_0 & * \\ \mathbb O & T_1\end{pmatrix}$$
with respect to some decomposition  of $\mathcal H$, 
and there exists
$\mathcal M_1\in\operatorname{Lat}T_1$ such that $T_1|_{\mathcal M_1}\sim S$. Then there exists
$\mathcal M\in\operatorname{Lat}T$ such that $T|_{\mathcal M}\sim S$. 
\end{lemma}
\begin{proof} Let $x$ be a cyclic vector for $T_1|_{\mathcal M_1}$. Set $\mathcal M=\vee_{n\geq 0}T^nx$. 
It is easy to see that $P_{\mathcal M_1}|_{\mathcal M}$ realizes the relation 
$T|_{\mathcal M}\buildrel d \over\prec T_1|_{\mathcal M_1}$. By Lemma \ref{lemcycss}, 
$T|_{\mathcal M}\sim S$. 
\end{proof}

 \begin{lemma}\label{lemdense} Suppose that $T\in\mathcal L(\mathcal H)$ is an a.c. polynomially bounded operator, 
 $\sigma\subset\mathbb T$ is a measurable set, 
$X\in\mathcal L(\mathcal H,L^2(\sigma))$, $XT=U_\sigma X$ and 
$\bigvee_{k\geq 0}U_\sigma^{-k} X\mathcal H = L^2(\sigma)$.
Then there exists $x\in\mathcal H$ such that $Xx=:f\neq 0$ a.e. on $\sigma$.
\end{lemma}

\begin{proof} Let $\{x_n\}_{n=1}^N$ be such that $x_n\in\mathcal H$, $\|x_n\|=1$ for all $n$, and 
$$\bigvee_{n=1}^N\bigvee_{k\geq 0}T^kx_n=\mathcal H,$$
where $1\leq N\leq\infty$. Set $f_n=Xx_n$. 
Then $$m\Bigl(\bigcap_{n=1}^N\{\zeta\in\sigma\ :\ f_n(\zeta)=0\}\Bigr)=0.$$ 
Take $\{a_n\}_{n=1}^N$ such that $a_n> 0$ for all $n$, 
and $\sum_{n=1}^Na_n<\infty$. Set $$\sigma_{1n}=\{\zeta\in\sigma\ :\ |f_n(\zeta)|>1\}.$$
Let $\{\varphi_n\}_{n=1}^N\subset H^\infty$ be such that 
$$ |\varphi_n|=\begin{cases}  a_n  & \text{ on }  \mathbb T\setminus\sigma_{1n}, \\
\frac{a_n}{|f_n|}& \text{ on } \sigma_{1n}.\end{cases}$$
Set $x_{1n}=\varphi_n(T)x_n$ and $f_{1n}=Xx_{1n}$. Then $\|x_{1n}\|\leq C_{{\rm pol}, T}a_n$ and 
$|f_{1n}|\leq a_n$ a.e. on $\sigma$. Set $f_0=\sup_n|f_{1n}|$. Then $0<f_0<\infty$ a.e. on $\sigma$. 
Take $0<\delta<1$. Set
$$ \sigma_{2n}=\{\zeta\in\sigma\ :\ \delta f_0(\zeta)\leq |f_{1n}(\zeta)|\}.$$
Set $$\tau_1=\sigma_{21}, \ \  \tau_n=\sigma_{2n}\setminus\cup_{k=1}^{n-1}\tau_k, \ \ n\geq 2.$$
Then $$\bigcup_{n=1}^N\tau_n=\sigma \ \ \text{ and } \ \tau_n\cap\tau_k=\emptyset, 
\ \text{ if } \ n\neq k, \ n,k\geq 1.$$
Take $\{\varepsilon_n\}_{n=1}^N$ such that $\varepsilon_n>0$ for all $n$, and $\sum_{n=1}^N\varepsilon_n<\delta$. 
Let $\{\eta_n\}_{n=1}^N\subset H^\infty$ be such that 
$$ |\eta_n|=\begin{cases} 1 & \text{ on } \tau_n, \\
\varepsilon_n & \text{ on } \mathbb T\setminus\tau_n.\end{cases}$$
Put $$ x= \sum_{n=1}^N \eta_n(T)x_{1n}.$$ Then
$$f= \sum_{n=1}^N \eta_n f_{1n}.$$
We have $$|f|\geq|f_{1n}|-\sum_{k\neq n}\varepsilon_k|f_{1k}|\geq\Bigl(\delta-\sum_{k=1}^N\varepsilon_k\Bigr)f_0 
\ \ \text{ a.e.  on } \ \tau_n.$$ 
Thus,  $x$ satisfies the conclusion of the lemma.
\end{proof}

\begin{theorem}\label{thm66} Suppose that $T_1$ is a polynomially bounded operator such that 
$T_1\sim S$. Then for every $c>0$ there exist $\mathcal M\in \operatorname{Lat}T_1$ 
and $W\in \mathcal L(H^2,\mathcal M)$ such that $W$ is invertible, $WS=T_1|_{\mathcal M}W$ and $\|W\|\|W^{-1}\|\leq (1+c)\bigl(\sqrt 2(K^2+2)K C_{{\rm pol}, T_1}+1\bigr)\sqrt{ K^2 C_{{\rm pol}, T_1}^2 + 1} K  C_{{\rm pol}, T_1}^2$. 
\end{theorem}

\begin{proof} Applying  {\cite[Theorem 1]{ker89}}, it is easy to see that $(T_1)^{(a)}_+\cong S$. Denote by $X_+$ the canonical intertwining 
mapping for $T_1$ and $S$ constructed using  a Banach limit. Denote by $\mathcal H_1$ the space in which $T_1$ acts. 
Let $\frac{1}{1+c}\leq c_3<c_2<1$. By Lemma \ref{asymp1}, there exists $x_1\in\mathcal H_1$ such that $\|x_1\|=1$ and $\|X_+x_1\|\geq c_2$. 
Set $f_1=X_+x_1$. 
Denote by $\psi_{x_1}$ the function from Corollary \ref{tha} applied to $T_1$ and $x_1$. 
By Lemma \ref{lemfpsi}, 
\begin{equation}\label{estg1}|f_1|\leq \|X_+\|K C_{{\rm pol}, T_1}| \psi_{x_1}| \ \ \text{ a.e. on } \mathbb T.\end{equation}
Since $f_1\in H^2$, we have $\int_{\mathbb T}\log| \psi_{x_1}| \text{\rm d}m>-\infty$. 
Taking into account Remark \ref{h2}, we can assume that $\psi_{x_1}$ is outer. Put $g_1=f_1/\psi_{x_1}$. Then $g_1\in H^\infty$. 

Acting as in Proposition \ref{proptau1}, we obtain that 
\begin{equation}\label{gnorma} \|g_1\|_\infty > c_3. \end{equation}
Namely, if $|f_1|\leq c_3|\psi_{x_1}|$ a.e.  on $\mathbb T$, 
then $$c_2^2\leq \|X_+x_1\|^2=\|f_1\|^2 = \int_{\mathbb T}|f_1|^2\text{\rm d}m\leq \int_{\mathbb T}c_3^2| \psi_{x_1}|^2\text{\rm d}m\leq c_3^2,$$
because $\int_{\mathbb T}| \psi_{x_1}|^2\text{\rm d}m\leq 1$, a contradiction with the choice of $c_2$ and $c_3$. 

Set $\mathcal N = \bigvee_{k\geq 0}T_1^kx_1$,  $T_0=T_1|_{\mathcal N}$, and  $X_0=X_+|_{\mathcal N}$. 
Taking into account that $S|_{\operatorname{clos}X_+\mathcal N}\cong S$ and deleting on the inner factor of $f_1$, 
we can assume that $g_1$ is outer and we accept that $X_0$ realizes the relation $T_0\prec S$.  
 By  Corollary \ref{tha},  Remark \ref{h2},  and Lemma \ref{lemquasi} there exists  
$Y_0\in\mathcal L(H^2, \mathcal N)$ such that 
$\|Y_0\|\leq K C_{{\rm pol}, T_1},$ $Y_0$ realizes the relation $S\prec T_0$,   and $X_0Y_0=g_1(S)$.  

Clearly  $T_0\in\mathcal L(\mathcal N)$ satisfies the assumption of Theorem \ref{thmmain}. We prove Theorem \ref{thmmain} for $T_0$ and obtain 
estimates of $\|W\|$ and $\|Z\|$ for $W$ and $Z$  constructed in the proof of  Theorem \ref{thmmain}.

If $|g_1|\geq c_3$ a.e. on $\mathbb T$, then set $W=Y_0$. We have $W^{-1}=\frac{1}{g_1}(S)X_0$  and  
$\|W\|\|W^{-1}\|\leq(1+c)K C_{{\rm pol}, T_1}^2$. 
 Therefore, we assume that 
\begin{equation}\label{vybornew1} m(\{\zeta\in\mathbb T \ :\ |g_1(\zeta)|\geq c_3\})<1. \end{equation}

Applying  Proposition  \ref{proptt0tt} to $T_0$, $Y_0$,  $X_0$, and $g_1$,  we obtain   
$$T\in\mathcal L(\mathcal N\oplus H^2_-) \ \ \text{ and } \ \ \ X\in\mathcal L(\mathcal N\oplus H^2_-, L^2(\mathbb T)).$$ 
Taking into account the estimates of $\|Y_0\|$, of $\|g_1\|_\infty$ by \eqref{estg1}, and evident inequality $C_{{\rm pol}, T_0}\leq C_{{\rm pol}, T_1}$,
we obtain that  
\begin{equation}\label{estxxtt} C_{{\rm pol}, T}\leq \sqrt{K^2+2}C_{{\rm pol}, T_1}\ \ \text{ and }
\ \ \|X\|\leq C_{{\rm pol}, T_1}\sqrt{ K^2 C_{{\rm pol}, T_1}^2 +1}.
\end{equation}

Since 
$T_0$ is of class $C_{10}$, we have $\tau=\emptyset$, $X_*^*=Y_0\oplus I_{H^2_-}$, and $XX_*^*=g_1(U_{\mathbb T})$. 
Therefore, $g=g_1$ and $h_{01}=h_0$, where $g$  and $h_{01}$ are defined in \eqref{gnew} and \eqref{h01new}, respectively ($h_0$ will be chosen below; $g$ and $h_{01}$ will be used in \eqref{alphanew}). We replace \eqref{vyborsigma} by 
 \begin{equation}\label{vybornew} \sigma = \{\zeta\in\mathbb T \ :\ |g_1(\zeta)|<c_3\}. \end{equation}
It allows to improve the estimate of $\|Z\|$ constructed in the end of Theorem \ref{thmmain}.
By \eqref{vybornew1} and \eqref{gnorma}, $0<m(\sigma)<1$. 

Set $h_0(\zeta)=1$ ($\zeta\in\mathbb T$). 
Define $x_0$ as in \eqref{x0ppsigma}.   Namely, set $x_0=X_*^*\chi_\sigma$. 

Take $0<\varepsilon_{00},\gamma_0,\delta_0<1$ and apply Proposition \ref{proptau2} as in the proof of Theorem \ref{thmmain}. 
The first inequality in \eqref{vyborsigma} and the choice of $\delta_0$ in the proof of Theorem \ref{thmmain}
was used to apply Lemmas \ref{ker89thm3} and \ref{lemast} and obtain that $\psi_n$ constructed by Corollary \ref{tha} applied to $\varphi_n(T)x_0$ 
satisfy \eqref{lognew}. In the case $T_0\sim S$ considered now  \eqref{phifleq} implies 
$$\int_\sigma\log| \psi_n| \text{\rm d}m>-\infty \ \text{ for every } n\in\mathbb N.$$  
Take $t_0>0$ and  set 
$$|\psi_{0n}|=\begin{cases}|\psi_n| & \text{ on } \sigma, \\
\max(|\psi_n|, t_0) & \text{ on }  \mathbb T\setminus\sigma.\end{cases}$$
Then $\psi_{0n}$ satisfy \eqref{lognew} for all $n\in\mathbb N$.
Repeat the remaining part of the proof  of Theorem \ref{thmmain} with $\psi_{0n}$ instead of $\psi_n$. 
We obtain the constants $C_1$, $C_2$, the function $\alpha$ defined as in \eqref{alphanew} which satisfies \eqref{alphainvtt} and intertwining transformations $W$ and $Z$ with the following estimates:
$$ C_1=\frac{\|X_*^*\|}{KC_{{\rm pol}, T}(1+\varepsilon_0)}\leq
\frac{1}{(1+\varepsilon_0)},$$
because $\|X_*\|\leq K C_{{\rm pol}, T_1}\leq K C_{{\rm pol}, T}$;
\begin{align*}C_2 &
 \leq 
K^2C_{{\rm pol}, T}^2(1+\varepsilon_0)\frac{(2(1+c_1^2))^{1/2}}{1-c_1}
\max\bigl((1+\sum_{n\geq 2}\xi_n^2)C_1, 1+\xi\bigr)  \\ & \ \ \ \ +
KC_{{\rm pol}, T_1}\bigl(1+\sum_{n\geq 2}\xi_n^2\bigr) ;
\end{align*}
$$\|W\|\leq C_2; \ \ \ \ \|Z\|\leq \Bigl\|\frac{1}{\alpha}\Bigr\|_\infty\|X\|.$$
Choosing $c_1$, $\varepsilon_{00}$, $\{\xi_n\}_{n\in\mathbb N}$ close to $0$ and $c_3$, $\gamma_0$, $\delta_0$ close to $1$, 
taking into account that $\varepsilon_0\leq\varepsilon_{00}$, the definition \eqref{xi2} of $\xi$, the estimates before \eqref{alphainvsigma} and \eqref{alphainvsmsigma}, 
the choice of $\sigma$ by \eqref{vybornew}, and \eqref{estxxtt}, 
we obtain the conclusion of the theorem.\end{proof}

\begin{remark} There exist operators $T$ which satisfy the assumptions of Theorem  \ref{thmmain} and have
non-trivial factorization of the form $\begin{pmatrix} C_{\cdot 1} & * \\ \mathbb O & C_{\cdot 0} \end{pmatrix}$, see 
 {\cite[Sec. IX.2]{sfbk}}. Therefore,  to obtain estimates of the norms
of intertwining transformations in terms of $C_{{\rm pol},T}$, a separate consideration of operators of class $C_{10}$ is needed in the present construction.
\end{remark}

{\bf Proof of Corollary \ref{cormain}.} We have $T=T_a\dotplus T_s$, where $T_a$ is a.c. and $T_s$ is similar to 
a singular unitary operator, see \cite{mlak} or \cite{ker16}. Since there is no nonzero transformation intertwining a.c. and singular unitaries, we conclude that $T_a\buildrel d \over\prec  U_{\mathbb T}$. Let $x$ be from  Lemma \ref{lemdense} 
applied to $T_a$ (with $\sigma=\mathbb T$). Set $$\mathcal N_0=\bigvee_{k\geq 0}T_a^kx.$$ 
Then $\mathcal N_0\in\operatorname{Lat}T$, $T|_{\mathcal N_0}$ is a  cyclic  a.c. polynomially bounded operator and 
 $T|_{\mathcal N_0}\buildrel d \over\prec  U_{\mathbb T}$. By Lemma \ref{lemcyc},  
$\bigl(T|_{\mathcal N_0}\bigr)^{(a)}\cong U_{\mathbb T}$. Let 
$$T|_{\mathcal N_0}=\begin{pmatrix}T_0 & * \\ \mathbb O & T_1\end{pmatrix}, $$
where $T_0$ and $T_1$ are of  classes $C_{0\cdot}$ and 
$C_{1\cdot}$, respectively, see Introduction and \cite{ker89}.  By  {\cite[Theorem 3]{ker89}}, 
$T_1^{(a)}\cong \bigl(T|_{\mathcal N_0}\bigr)^{(a)}$.
Since   $T|_{\mathcal N_0}$ is   cyclic and   a.c., we have $T_1$  is   cyclic and   a.c., too. 
By Theorem \ref{thmmain} applied to $T_1$, there exists $\mathcal M_1\in\operatorname{Lat}T_1$ 
such that $T_1|_{\mathcal M_1}\approx S$. 
 By Lemma \ref{proptrian},  there exists 
$\mathcal M\in\operatorname{Lat}T$ such that $T|_{\mathcal M}\sim S$. It remains to apply Theorem \ref{thm66} to $T|_{\mathcal M}$.
\qed

\end{document}